%% file: paper.tex
\documentclass[12pt]{amsart}

\usepackage{amsmath}    
\usepackage{amsfonts}    
\usepackage{amssymb}
\usepackage{graphicx}   
\usepackage{subfigure}  
\usepackage{hyperref}   
\usepackage{amsthm}
\usepackage{cite}
\usepackage{enumerate}
\usepackage{bbm}
\usepackage{mathtools}
\usepackage{float}
\usepackage{comment}
\usepackage{bbm}
\usepackage{dsfont}

\renewcommand{\d}{\mathrm{d}}
\renewcommand{\div}{\mathrm{div}}

\newcommand{\nb}{\overline{\nabla}}

\numberwithin{equation}{section}

\author[T. Begley]{Tom Begley}
\address{Tom Begley.
CCA, Centre for Mathematical Sciences,
Wilberforce Road,
Cambridge,
CB3 0WA,
United Kingdom.}
\email{T.Begley@maths.cam.ac.uk}

\author[K. Moore]{Kim Moore}
\thanks{Both authors supported by EPSRC grant EP/H023348/1 at the Cambridge Centre for Analysis}
\address{Kim Moore. 
CCA, Centre for Mathematical Sciences,
Wilberforce Road,
Cambridge,
CB3 0WA,
United Kingdom.}
\email{K.Moore@maths.cam.ac.uk}

\usepackage[left=8em,top=6em,right=8em]{geometry}

\newtheorem{theorem}{Theorem}[section]
\newtheorem{lemma}[theorem]{Lemma}

\newtheorem{corollary}[theorem]{Corollary}
\newtheorem*{theorem*}{Theorem}
\newtheorem*{lemma*}{Lemma}
\newtheorem*{proposition*}{Proposition}
\newtheorem*{corollary*}{Corollary}

\theoremstyle{definition}

\newtheorem*{remark}{Remark}

\title{On short time existence of Lagrangian mean curvature flow}

\begin{document}
\begin{abstract}
We consider a short time existence problem motivated by a conjecture of Joyce in \cite{JoyceCon}. Specifically we prove that given any compact Lagrangian $L\subset \mathbb{C}^n$ with a finite number of singularities, each asymptotic to a pair of non-area-minimising, transversally intersecting Lagrangian planes, there is a smooth Lagrangian mean curvature flow existing for some positive time, that attains $L$ as $t \searrow 0$ as varifolds, and smoothly locally away from the singularities.
\end{abstract}
\maketitle
\thispagestyle{empty}
\section{Introduction}
\input{intro}

\section{Preliminaries}
\input{defs}

\section{Evolution Equations and Monotonicity Formulas}\label{sec:eveq}
\input{eveq}

\section{Stability of Self-Expanders}\label{sec:stability}
\input{stability}

\section{Main Theorem}\label{sec:main}
\input{mainthm}

\section{Short-time Existence}
\input{shorttime}

\section{Construction of Approximating Family}\label{sec:approx}
\input{approx}

\section{Appendix} \label{sec:appendix}
\input{appendix}

\bibliography{bib}
\bibliographystyle{hplain}

\end{document}

%% file: intro.tex
\noindent A long-standing open problem in the study of Calabi-Yau manifolds is whether given a Lagrangian submanifold, one can find a special Lagrangian in its homology or Hamiltonian isotopy class. Special Lagrangians are always area minimising, so one way to approach the existence problem is to try to minimise area among all Lagrangians in a given class. This minimisation problem turns out to be very subtle and fraught with difficulties. Indeed Schoen-Wolfson \cite{schwo} showed that when the real dimension is $4$, given a particular class one can find a Lagrangian minimising area among Lagrangians in that class, but that the minimiser need not be a special Lagrangian. Later Wolfson\cite{wolf} found a $K3$ surface and a Lagrangian sphere in this surface such that the area minimiser among Lagrangians in the homology class of the sphere, is not special Lagrangian, and the area minimiser in the class is not Lagrangian. \\
An alternative way of approaching the problem is to consider mean curvature flow. Mean curvature flow is a geometric evolution of submanifolds where the velocity at any point is given by the mean curvature vector. This can also be seen as the gradient descent for the area functional. Smoczyk showed in \cite{Smoc96} that the Lagrangian condition is preserved by mean curvature flow if the ambient space is K\"ahler-Einstein, and consequently mean curvature flow has been proposed as a means of constructing special Lagrangians.  In order to flow to a special Lagrangian, one would need to show that the flow exists for all time. This however can't be expected in general, as finite time singularities abound. See for example Neves \cite{Nev12}. For a nice overview on what is known about singularities of Lagrangian mean curvature flow, we refer the reader to the survey paper of Neves \cite{nevessurvey}.\\
A natural question is whether it might be possible to continue the flow in a weaker sense once a singularity develops and, in doing so, to push through the singularity. Since all special Lagrangians are zero-Maslov class, and the Maslov class is preserved by Lagrangian mean curvature flow, of particular interest is the mean curvature flow of zero-Maslov class Lagrangians. In this case, the structure of singularities is relatively well understood. Indeed Neves \cite{Nev06} has shown that a singularity of zero-Maslov class Lagrangian mean curvature flow must be asymptotic to a union of special Lagrangian cones. We note that in $\mathbb{C}^2$ every such union is simply a union of Lagrangian planes, and so the case we consider in the below theorem is not necessarily overly restrictive. In this paper we consider the simplest such singularity, namely that where the singularities are each asymptotic to the union of two non-area-minimising, transversally intersecting Lagrangian planes. Specifically we prove the following theorem which serves as a partial answer to Problem 3.14 in \cite{JoyceCon}.
\begin{theorem*}[Short-time Existence]
Suppose that $L\subset\mathbb{C}^n$ is a compact Lagrangian submanifold of $\mathbb{C}^n$ with a finite number of singularities, each of which is asymptotic to a pair of transversally intersecting planes $P_1 + P_2$ where neither $P_1 + P_2$ nor $P_1 - P_2$ are area minimizing. Then there exists $T > 0$ and a Lagrangian mean curvature flow $(L_t)_{0<t<T}$ such that as $t \searrow 0$, $L_t \rightarrow L$ as varifolds and in $C^\infty_{loc}$ away from the singularities.
\end{theorem*}
We remark that the assumptions $L \subset \mathbb{C}^n$ and $L$ compact are made to simplify the analysis in the sequel, however since the analysis is all of an entirely local nature we may relax this to $L \subset M$ for some Calabi-Yau manifold $M$, and to $L$ non-compact provided, in the latter case, that we impose suitable conditions at infinity.\\
In the one-dimensional case all curves are Lagrangian. Ilmanen-Neves-Schulze considered the flow of planar networks, that is finite unions of embedded line segments of non-zero length meeting only at their endpoints, in \cite{IlNeSc14}. They showed that there exists a flow of regular networks, that is networks where at any meeting point exactly three line segments come together at angles of $2\pi/3$, starting at any initial non-regular network. To do so they performed a gluing procedure to get an approximating family of regular initial conditions, and proved uniform estimates on the corresponding flows, allowing them to pass to a limit of flows to prove the result. The proof here is based heavily on their arguments, and many of the calculations we do are similar to those in that paper. To prove the short-time existence, we construct a smooth approximating family $L^s$ of initial conditions via a surgery procedure. Specifically we take a singularity asymptotic to some non-area-minimising pair of planes $P_1 + P_2$, cut it out and glue in a piece of the Lagrangian self-expander asymptotic to those planes at a scale determined by $s$. For full details see Section 7. Each of these approximating Lagrangians is smooth, and hence standard short time existence theory gives a smooth Lagrangian mean curvature flow $L^s_t$ corresponding to each $s$. As $s\rightarrow 0$ the curvature of $L^s$ blows up so the existence time of the flows $L^s_t$ guaranteed by the standard short time existence theory goes to zero. Instead we are able to prove uniform estimates on the Gaussian density ratios of $L^s_t$, which combined with the local regularity result of Brian White \cite{Whi05} provides uniform curvature estimates, interior in time, on the flows $L^s_t$, from which we obtain a uniform time of existence allowing us to pass to a limit of flows and prove the main result.\\
There are two key components in the proof of the estimates on the Gaussian density ratios. The first is a stability result for self-expanding solutions to Lagrangian mean curvature flow. More specifically we show that if a Lagrangian is weakly close to a Lagrangian self expander in an $L^2$ sense, then it is close in a stronger $C^{1,\alpha}$ sense. The proof of this stability result depends crucially on a uniqueness result for zero-Maslov smooth self-expanders asymptotic to transverse pairs of planes due to Lotay-Neves \cite{LotNe13} and Imagi-Joyce-Oliveira dos Santos \cite{ImJoOl14}. The second component is a monotonicity formula for the self-expander equation, which allows us to show that the approximating family of initial conditions that we construct in the proof, which are self-expanders in a ball, remain weakly close to the self-expander for a short time. The combination of these results tells us that the evolution of the approximating flows is close to the evolution of the self-expander near the singularity. Since self-expanders move by dilation, we have good curvature control on the self-expander, and hence estimates on the Gaussian densities of the approximating flow.\\
\emph{Organisation.} The paper is organised as follows. In Section 2 we recall key definitions and results. In Section 3 we derive evolution equations and monotonicity formulas for geometric quantities under the flow. In Section 4 we prove the Stability result mentioned above. Section 5 contains the proof of the main theorem which gives uniform estimates on the Gaussian density ratios of the approximating family near the singularity. From this we get uniform estimates, interior in time, on the curvature of the approximating family which allows us to appeal to a compactness argument. Section 6 contains the proof of the short time existence result itself. Section 7 details the construction of the approximating family used in the proof of the main theorem. Finally the appendix, Section 8, contains miscellaneous technical results, including Ecker-Huisken style curvature estimates for high-codimension mean curvature flow.\\
\emph{Acknowledgements.} Both authors would like to thank Jason Lotay and Felix Schulze for their help, guidance and feedback.

%% file: defs.tex
\subsection{Mean Curvature Flow}
Let $M^n \subset \mathbb{R}^{n+k}$ be an $n$-dimensional embedded submanifold of $\mathbb{R}^{n+k}$. A mean curvature flow is a one parameter family of immersions $F:M\times[0,T)\rightarrow\mathbb{R}^{n+k}$ such that the normal velocity at any point is given by the mean curvature vector, that is
\begin{equation*}
	\frac{dF}{dt} = \vec{H}.
\end{equation*}
Of particular interest to us are \emph{self-expanders}. These are submanifolds $M \subset \mathbb{R}^{n+k}$ satisfying the elliptic equation
\begin{equation*}
	\vec{H} - x^\perp = 0
\end{equation*}
where $(\cdot)^\perp$ is the projection to the normal space. In this case one can show that $M_t = \sqrt{2t}M$ is a solution of mean curvature flow. 
\\
A fundamental tool in the analysis of mean curvature flow is the Gaussian density. We first define the backwards heat kernel $\rho_{(x_0,t_0)}$ as follows
\begin{equation*}
	\rho_{(x_0,t_0)}(x,t) := \frac{1}{(4\pi(t_0 - t))^{n/2}}\exp\left(-\frac{|x-x_0|^2}{4(t_0 - t)}\right).
\end{equation*}
Next, for a mean curvature flow $(M_t)_{0\leq t<T)}$ we define the \emph{Gaussian density ratio} centred at $(x_0,t_0)$ and at scale $r$ by
\begin{align*}
	\Theta(x_0,t_0,r) :&= \int_{M_{t_0 - r^2}} \rho_{(x_0,t_0)}(x,t_0 - r^2)d\mathcal{H}^n(x) \\
	&= \int_{M_{t_0 - r^2}} \frac{1}{(4\pi r^2)^{n/2}}\exp\left(-\frac{|x-x_0|^2}{4r^2}\right)d\mathcal{H}^n(x)
\end{align*}
this is defined for $0 < t_0 \leq T$, $0 < r \leq \sqrt{t_0}$ and any $x_0 \in \mathbb{R}^{n+k}$. Huisken in \cite{Hui90} proved the following monotonicity formula.
\begin{theorem}[Monotonicity Formula]
If $(M_t)_{0\leq t<t_0}$ is a mean curvature flow, then
\begin{equation*}
	\frac{d}{dt}\int_{M_t} \rho_{(x_0,t_0)}(x,t)d\mathcal{H}^n(x) = -\int_{M_t} \left|\vec{H} - \frac{(x_0 - x)^\perp}{2(t_0 - t)}\right|^2\rho_{(x_0,t_0)}(x,t)d\mathcal{H}^n(x).
\end{equation*}
\end{theorem}
In particular, it follows that $\Theta(x_0,t_0,r)$ is non-decreasing in $r$. Consequently we can define the \emph{Gaussian density} as
\begin{equation*}
	\Theta(x_0,t_0) := \lim_{r\searrow 0} \Theta(x_0,t_0,r).
\end{equation*}
One can show that $(x_0,t_0)$ is a regular point of the flow if and only if $\Theta(x_0,t_0) = 1$. The following local regularity theorem of White \cite{Whi05} says that if the density ratios are close to $1$, then that is enough to get curvature estimates.
\begin{theorem}[Local regularity]
Let $\tau > 0$. There are $\varepsilon_0(n,k) > 0$ and $C = C(n,k,\tau) < \infty$ such that if $\partial M_t \cap B_{2r} = \emptyset$ for $t \in [0,r^2)$ and
\begin{equation*}
	\Theta(x,t,\rho) \leq 1 + \varepsilon_0  \quad\quad \rho \leq \tau\sqrt{t}, \: x\in B_{2r}(x_0), \: t \in [0,r^2)
\end{equation*}
Then
\begin{equation*}
	|A|(x,t) \leq \frac{C}{\sqrt{t}} \quad\quad x \in B_r(x_0), \: t\in [0,r^2).
\end{equation*}
\end{theorem}
Finally, we introduce what it means for two manifolds to be \emph{$\varepsilon$-close in $C^{1,\alpha}$}. Given an open set $U$ and two n-dimensional manifolds $\Sigma$ and $L$ defined in $U$, we say that $\Sigma$ and $L$ are $1$-close in $C^{1,\alpha}(W)$ for any $W$ with $\mathrm{dist}(W, \partial U) \geq 1$ if for all $x \in W$, $B_1(x) \cap \Sigma$ and $B_1(x) \cap L$ are both graphical over some common $n$-dimensional plane, and if $u$ and $v$ denote the respective graph functions then $\|u - v\|_{1,\alpha} \leq 1$. We then say that $\Sigma$ and $L$ are $\varepsilon$-close in $W$ if after rescaling by a factor $1/\varepsilon$, $\Sigma$ and $L$ are $1$-close in $\varepsilon^{-1}W$ for any $W$ with $\mathrm{dist}(\varepsilon^{-1}W, \varepsilon^{-1}\partial U) \geq 1$.

\subsection{Lagrangian Submanifolds and Lagrangian Mean Curvature Flow}
We consider $\mathbb{C}^n$ with the standard complex coordinates $z_j = x_j + iy_j$. We will often identify $\mathbb{C}^n$ with $\mathbb{R}^{2n}$. We let $J$ denote the standard complex structure on $\mathbb{C}^n$ and $\omega$ the standard symplectic form on $\mathbb{C}^n$. We say that a smooth $n$-dimensional submanifold of $\mathbb{C}^n$ is \emph{Lagrangian} if $\omega|_L = 0$. We also consider the closed $n$-form $\Omega$, called the \emph{holomorphic volume form}, defined by
\begin{equation*} 
	\Omega := dz_1\wedge\dots\wedge dz_n.
\end{equation*}
On any oriented Lagrangian a simple computation shows that $\Omega|_L = e^{i\theta_L} \mathrm{vol}_L$, where $\mathrm{vol}_L$ is the volume form on $L$. $e^{i\theta_L}:L\rightarrow S^1$ is called the \emph{Lagrangian phase}. $\theta_L$ is called the \emph{Lagrangian angle}, and may be a multi-valued function. We henceforth suppress the subscript $L$. In the case that $\theta$ is a single valued function, we say that the Lagrangian $L$ is \emph{zero-Maslov}. An equivalent condition is $[d\theta] = 0$, that is, $d\theta$ is cohomologous to $0$. If $\theta \equiv \theta_0$ is constant, then we say that $L$ is \emph{special Lagrangian}. In this case $L$ is calibrated by $\mathrm{Re}(e^{-i\theta_0}\mathrm{vol}_L)$, and hence is area-minimising in its homology class.\\
We also consider the Liouville form $\lambda$ on $\mathbb{C}^n$ defined by
\begin{equation*}
	\lambda := \sum_{j=1}^n x_jdy_j - y_jdx_j.
\end{equation*}
A simple calculation verifies that $d\lambda = 2\omega$. If there exists some function $\beta$ such that $\lambda|_L = d\beta$ then we say that $L$ is \emph{exact}. In this paper we will be more interested in local exactness, that is when the Liouville form $\lambda$ only has a primitive in some open set.\\
The following remarkable property of smooth Lagrangians relates the Lagrangian angle and mean curvature vector (see for example \cite{ThoYau02})
\begin{equation*}
	\vec{H} = J\nabla\theta.
\end{equation*}
Consequently we see that the smooth minimal Lagrangians are exactly the smooth special Lagrangians.\\
A \emph{Lagrangian mean curvature flow} is a mean curvature flow $(L_t)_{0\leq t<T}$ with $L_0$ Lagrangian. As proved by Smoczyk \cite{Smoc96}, it turns out that the Lagrangian condition is preserved by the mean curvature flow.

%% file: eveq.tex
In this section we compute evolution equations for different geometric quantities under the flow, and then use these to prove a local monotonicity formula for a primitive of the expander equation.
\begin{lemma}\label{lem:eveq}
The following evolution equations hold.
\begin{itemize}
\item[(i)] $\frac{d\theta_t}{d t} = \Delta \theta_t$
\item[(ii)] In an open set where the flow is exact and zero-Maslov $\frac{d\beta_t}{d t} = \Delta \beta_t - 2\theta_t$
\item[(iii)] $\frac{d \rho_{(x_0,t_0)}}{d t} = - \Delta\rho_{(x_0,t_0)} - \left|\vec{H} - \frac{(x_0 - x)^\perp}{2(t_0 -t)}\right|^2\rho_{(x_0,t_0)} + H^2\rho_{(x_0,t_0)}$.
\end{itemize}
\end{lemma}
\begin{proof}
(i) Differentiating the holomorphic volume form $\Omega$ and using Cartan's formula we have
\begin{align*}
	\frac{d\Omega}{d t} = \mathcal{L}_{\vec{H}}\Omega &= d(\vec{H} \lrcorner \Omega) = d(ie^{i\theta_t}\nabla\theta_t \lrcorner \mathrm{vol}_{L_t}) \\
	&= ie^{i\theta_t}d(\nabla\theta_t  \lrcorner \mathrm{vol}_{L_t}) - e^{i\theta_t}d\theta_t\wedge(\nabla\theta_t \lrcorner  \mathrm{vol}_{L_t}) \\
	&= ie^{i\theta_t}\div(\nabla\theta_t)\mathrm{vol}_{L_t} - e^{i\theta_t}d\theta_t\wedge(\nabla\theta_t \lrcorner  \mathrm{vol}_{L_t}).
\end{align*}
On the other hand
\begin{align*}
	\frac{d\Omega}{d t} = \frac{d}{d t}\left(e^{i\theta_t}\mathrm{vol}_{L_t}\right) = ie^{i\theta_t}\frac{d\theta_t}{d t}\mathrm{vol}_{L_t} + e^{i\theta_t}\frac{d}{d t}\mathrm{vol}_{L_t}.
\end{align*}
Comparing real and imaginary parts we have (i). \\
(ii) Using Cartan's formula again and denoting $\lambda_t := \lambda|_{L_t}$, where $\lambda = \sum_{i=1}^n x^id y^i - y^i d x^i$ is the Liouville form, we have
\begin{align*}
	d\left(\frac{d\beta_t}{d t}\right) = \mathcal{L}_{\vec{H}}\lambda_t &= d(\vec{H}\lrcorner\lambda_t) + \vec{H}\lrcorner d\lambda_t \\
	&= d(\vec{H}\lrcorner \lambda_t) + J\nabla\theta_t \lrcorner 2\omega \\
	&= d(\vec{H}\lrcorner\lambda_t) - 2d\theta_t.
\end{align*}
Hence
\begin{equation*}
	d\left(\frac{d\beta_t}{d t} - \vec{H}\lrcorner\lambda_t + 2\theta_t\right) = 0.
\end{equation*}
By possibly adding a time-dependent constant to $\beta_t$ this implies
\begin{equation*}
	\frac{d\beta_t}{d t} = \vec{H}\lrcorner\lambda_t - 2\theta_t.
\end{equation*}
Hence it only remains to show that $\vec{H}\lrcorner\lambda_t = \Delta \beta_t$. We first show that $\nabla \beta_t = (Jx)^T$. Indeed we have $d\beta_t = \lambda_t$, thus for a tangent vector $\tau$
\begin{equation*}
	\langle \nabla \beta_t, \tau\rangle = d\beta_t(\tau) = \lambda_t(\tau) = \langle Jx,\tau \rangle = \langle (Jx)^T, \tau \rangle.
\end{equation*}
With this in hand we now choose normal coordinates at a point $x$, and denote the coordinate tangent vectors by $\{\partial_1,\dots,\partial_n\}$. Then we calculate
\begin{align*}
	\nabla_i\nabla_j\beta_t = \langle\nabla_i(Jx)^T, \partial_j\rangle &= \partial_i\langle Jx, \partial_j \rangle - \langle (Jx)^T, D_{\partial_i}\partial_j\rangle \\
	&= \langle J\partial_i, \partial_j\rangle + \langle Jx, D_{\partial_i}\partial_j\rangle - \langle (Jx)^T, D_{\partial_i}\partial_j \rangle \\
	&= \omega(\partial_i, \partial_j) + \langle (Jx)^\perp, D_{\partial_i}\partial_j \rangle \\
	&= \langle Jx, h_{ij} \rangle,
\end{align*}
where $h_{ij}$ is the second fundamental form. Taking the trace of each side we have
\begin{equation*}
	\Delta\beta_t = \langle Jx, \vec{H} \rangle = \vec{H}\lrcorner \lambda_t.
\end{equation*}
(iii) We may assume without loss of generality that $x_0 = 0$ and $t_0 = 0$, and we will suppress the subscripts of $\rho$. We first calculate
\begin{equation*}
	\frac{\partial\rho}{\partial t} = \left(-\frac{n}{2t} - \frac{|x|^2}{4t^2}\right)\rho \quad\quad\quad
	\frac{\partial \rho}{\partial x^i} = \frac{x^i}{2t}\rho\quad\quad\quad
	\frac{\partial^2\rho}{\partial x^i \partial x^j} = \left(\frac{\delta_{ij}}{2t}  + \frac{x^ix^j}{4t^2}\right)\rho.
\end{equation*}
Then we have
\begin{align*}
	\frac{d\rho}{d t} &= \frac{\partial\rho}{\partial t} + \langle D\rho, \vec{H} \rangle = \frac{\partial \rho}{\partial t} - \left\langle \frac{x^\perp}{2t}, \vec{H} \right\rangle\rho \\
	&= \frac{\partial\rho}{\partial t} - \left|\vec{H} - \frac{x^\perp}{2t}\right|^2\rho + H^2\rho + \frac{|x^\perp|^2}{4t^2}\rho.
\end{align*}
To compute the Laplacian term, we once again fix a point in $L$ and take normal coordinates at that point, with $\{\partial_1,\dots,\partial_n\}$ denoting the coordinate tangent vectors. Then
\begin{align*}
	\nabla_i \rho &= \frac{\langle x, \partial_i\rangle}{2t}\rho \\
	\nabla_j\nabla_i \rho &= \left(\frac{\delta_{ij}}{2t} + \frac{\langle x, \partial_i\rangle\langle x, \partial_j\rangle}{4t^2}\right)\rho 
\end{align*}
So we find that
\begin{equation*}
	\Delta\rho = \left(\frac{n}{2t} + \frac{|x^T|^2}{4t^2}\right)\rho = -\frac{\partial\rho}{\partial t} - \frac{|x^\perp|^2}{4t^2},
\end{equation*}
combining this with the previous calculation yields (iii).
\end{proof}
\begin{remark}
From the above evolution equations we see that local exactness is preserved by the flow, indeed
\begin{align*}
	\frac{d\lambda_t}{d t} = \mathcal{L}_{\vec{H}}\lambda_t &= d(\vec{H}\lrcorner\lambda_t) + \vec{H}\lrcorner d\lambda_T \\
	&= d(\vec{H}\lrcorner \lambda_t) + J\nabla\theta_t\lrcorner 2\omega \\
	&= d(\vec{H}\lrcorner\lambda_t) - 2d\theta_t.
\end{align*}
so by the fundamental theorem of calculus we have
\begin{equation*}
	\lambda_t = \lambda_0 + \int_0^t \frac{d\lambda_s}{ds}ds
\end{equation*}
where the right hand side is exact if $\lambda_0$ is.
\end{remark}
Let $\phi$ be a cut-off function supported on $B_3$ with $0\leq \phi \leq 1$, $\phi \equiv 1$ on $B_2$ and the estimates $|D\phi| \leq 2$ and $|D^2\phi| \leq C$. We then have the following lemma.
\begin{lemma}\label{lem:localmon}
Suppose that $(L_t)$ are exact in $B_3$ and define $\alpha_t := \beta_t + 2t\theta_t$. Then
\begin{equation*}
	\frac{d}{dt}\int_{L_t} \phi\alpha_t^2\rho d\mu \leq -\int_{L_t}\phi|2t\vec{H} - x^\perp|^2\rho d\mu + C\int_{L_t\cap(B_3\setminus B_2)} \alpha_t^2\rho d\mu.
\end{equation*}
where $C = C(\phi)$.
\end{lemma}
\begin{proof}
We calculate
\begin{equation*}
	\left(\frac{d}{dt} - \Delta_{L_t}\right)\phi = \frac{\partial\phi}{\partial t} - \div_{L_t} D\phi = -\Delta_{\mathbb{R}^{n+k}}\phi + \mathrm{tr}_{(TL)^\perp} D^2 \phi \leq C\mathbbm{1}_{B_3\setminus B_2}
\end{equation*}
Then
\begin{align*}
	\left(\frac{d}{dt} - \Delta\right) (\phi \alpha_t^2) &= \phi\left(\frac{d}{dt} - \Delta\right)\alpha_t^2 + \alpha_t^2\left(\frac{d}{dt} - \Delta\right)\phi - 2\langle \nabla\phi, \nabla\alpha_t^2\rangle \\
	&\leq 2\phi\alpha_t\left(\frac{d}{dt} - \Delta\right)\alpha_t - 2\phi|\nabla\alpha_t|^2 + C\alpha_t^2\mathbbm{1}_{B_3\setminus B_2} - 4\alpha_t \langle\nabla\phi, \nabla\alpha_t\rangle.
\end{align*}
Using Young's inequality we estimate the last term as follows
\begin{equation*}
	-4\alpha_t\langle \nabla\phi, \nabla\alpha_t\rangle \leq 4|D\phi||\alpha_t||\nabla\alpha_t| \leq \phi|\nabla\alpha_t|^2 + \frac{4|D\phi|^2}{\phi}\alpha_t^2 \leq \phi|\nabla\alpha_t|^2 + C\alpha_t^2\mathbbm{1}_{B_3\setminus B_2}
\end{equation*}
where we used that
\begin{equation*}
	\frac{|D\phi|^2}{\phi} \leq 2\max |D^2\phi| \leq C
\end{equation*}
which is true of any compactly supported smooth (or even $C^2$) function. Thus we arrive at
\begin{equation*}
	\left(\frac{d}{dt} - \Delta\right) \phi (\alpha_t)^2 \leq -\phi|\nabla\alpha_t|^2 + C\alpha_t^2\mathbbm{1}_{B_3\setminus B_2}.
\end{equation*}
We now just differentiate under the integral to get
\begin{align*}
	\frac{d}{dt}\int_{L_t} \phi\alpha_t^2\rho d\mu &\leq \int_{L_t} \phi\alpha_t^2\left(-\Delta\rho - \left|\vec{H} - \frac{x^\perp}{2t}\right|\rho + H^2\rho\right) - \rho\phi|\nabla\alpha_t|^2d\mu  \\
	&\quad\quad + \int_{L_t} \rho\Delta \phi\alpha_t^2 - \phi\alpha_t^2\rho H^2d\mu + C\int_{L_t\cap (B_3\setminus B_2)}\alpha_t^2\rho d\mu \\
	&\leq \int_{L_t} \rho\Delta\phi\alpha_t^2 - \phi\alpha_t^2\Delta\rho d\mu - \int_{L_t}\phi|\nabla \alpha_t|^2\rho d\mu + C\int_{L_t\cap(B_3\setminus B_2)} \alpha_t^2\rho d\mu.
\end{align*}
The first integral is zero by Green's identities, so we are left with precisely the desired inequality since $\nabla\alpha_t = \nabla\beta_t + 2t\nabla\theta_t = Jx^\perp - 2tJ\vec{H}$.
\end{proof}

%% file: stability.tex
In this section we prove a dynamic stability result for Lagrangian self-expanders. More specifically we show that if a Lagrangian submanifold is asymptotic to some pair of planes and is almost a self-expander in a weak sense, then the submanifold is actually close in a stronger topology to some self-expander. Let $P_1$, $P_2\subset\mathbb{C}^n$ be Lagrangian planes intersecting transversally such that neither $P_1 + P_2$ or $P_1 - P_2$ are area minimising. We denote by $P := P_1 + P_2$. We will need the following uniqueness result, proved by Lotay-Neves \cite{LotNe13} in dimension $2$ and Imagi-Joyce-Oliveira dos Santos \cite{ImJoOl14} in dimensions $3$ and higher.
\begin{theorem} \label{thm:seuniqueness}
There exists a unique smooth, zero-Maslov class Lagrangian self-expander asymptotic to $P$.
\end{theorem}
\begin{theorem} \label{thm:stability}
Fix $R,r,\tau > 0$, $\alpha$,$\varepsilon_0 < 1$, and $C, M < \infty$. Let $\Sigma$ be the unique smooth zero-Maslov Lagrangian self-expander asymptotic to $P$. Then for all $\varepsilon > 0$ there exists $\tilde R \geq R$, $\eta, \nu > 0$ each dependent on $\varepsilon_0$, $\varepsilon$, $r$, $R$, $\tau$, $\alpha$, $C$, $M$ and $P$ such that if $L$ is a smooth Lagrangian submanifold which is zero-Maslov in $B_{\tilde R}$
\begin{itemize}
\item[(i)] $|A|\leq M$  on  $L\cap B_{\tilde R}$
\item[(ii)] $\int_L \rho_{(x,0)}(y,-r^2)d\mathcal{H}^n \leq 1 + \varepsilon_0$ for all $x$ and $0 < r \leq \tau$,
\item[(iii)] $\int_{L\cap B_{\tilde R}} |\vec{H} - x^\perp|^2d\mathcal{H}^n \leq \eta$,
\item[(iv)] The connected components of $L\cap A(r,\tilde R)$ are in one to one correspondence with the connected components of $P \cap A(r, \tilde R)$ and
\begin{equation*}
	\mathrm{dist}(x, P) \leq \nu + C\mathrm{exp}\left(\frac{-|x|^2}{C}\right),
\end{equation*}
for all $x \in L\cap A(r, \tilde R)$;
\end{itemize}
then $L$ is $\varepsilon$ close to $\Sigma$ in $C^{1,\alpha}(\overline B_{\tilde R})$.
\end{theorem}
\begin{proof}
Seeking a contradiction, suppose that the result were not true. Then there would exist sequences $\nu_i \searrow 0$, $\eta_i \searrow 0$, $R_i \rightarrow \infty$ and $L_i$ such that each $L_i$ is a smooth Lagrangian submanifold of $\mathbb{C}^n$ that is zero-Maslov in $B_{R_i}$, satisfying
\begin{itemize}
\item[(1)] $|A^{L_i}| \leq M$ on $L_i\cap B_{R_i}$,
\item[(2)] $\int_{L_i} \rho_{(x,0)}(y,-r^2)d\mathcal{H}^n \leq 1 + \varepsilon_0$ for all $x$ and $0 < r \leq \tau$,
\item[(3)] $\int_{L_i\cap B_{R_i}} |\vec{H} - x^\perp|^2 d\mathcal{H}^n \leq \eta_i$
\item[(4)] The connected components of $L_i\cap A(r,R_i)$ are in one to one correspondence with the connected components of $P\cap A(r,R_i)$ and
\begin{equation*}
	\mathrm{dist}(x,P) \leq \nu_i + C\exp\left(\frac{-|x|^2}{C}\right)
\end{equation*}
for all $x \in L_i\cap A(r,R_i)$,
\item[(5)] $L_i$ is not $\varepsilon$-close to $\Sigma$ in $C^{1,\alpha}(B_{R_i})$.
\end{itemize}
By virtue of (1), (4), and a suitable interpolation inequality, it follows that for some $\rho > 0$, outside of $B_\rho$, $L_i$ and $\Sigma$ are both $\varepsilon/4$-close to $P$ in $C^{1,\alpha}$. Hence, in order that (5) is satisfied, we conclude that for large $i$, $L_i$ is not $\varepsilon$-close to $\Sigma$ in $C^{1,\alpha}(B_\rho)$. \\
On the other hand, by (1) and (2) we may extract a subsequence of $L_i$ that converges in $C^{1,\alpha}_{loc}$ for all $\alpha < 1$ to some limit $L_\infty$, a $C^{1,1}$ zero-Maslov Lagrangian submanifold. The estimate (2) passes to the limit and tells us that $L_\infty$ has unit multiplicity everywhere, and bounded area ratios. Since $L_\infty$ is $C^{1,1}$ we can define mean curvature in a weak sense, and (3) implies
\begin{equation*}
	\int_{L_\infty} |\vec{H} - x^\perp|^2 d\mathcal{H}^n = 0.
\end{equation*}
By standard Schauder theory for elliptic PDE, this immediately implies that $L_\infty$ is in fact smooth and satisfies the expander equation in the classical sense. Consequently $L_\infty$ is a smooth, zero-Maslov class Lagrangian submanifold, and (4) implies that $L_\infty$ is asymptotic to $P$. Theorem \ref{thm:seuniqueness} then implies that $L_\infty = \Sigma$, which contradicts (5).
\end{proof}

%% file: mainthm.tex
Suppose, as in the previous section, that $P := P_1 + P_2$ is a pair of transversely intersecting Lagrangian planes such that neither $P_1 + P_2$ nor $P_1 - P_2$ are minimising, and that $\Sigma$ is a Lagrangian self-expander asymptotic to $P$. We assume the existence of a family $(L^s)_{0<s\leq c}$ of compact Lagrangians, each exact and zero-Maslov in $B_4$ satisfying the following properties. The existence of such a family will be established in section \ref{sec:approx}
\begin{itemize}
\item[(H1)] The area ratios are uniformly bounded, i.e. there exists a constant $D_1$ such that
\begin{equation*}
	\mathcal{H}^n(L^s\cap B_r(x)) \leq D_1 r^n \quad \quad \forall r > 0, \: \forall s \in (0,c], \: \forall x.
\end{equation*}
\item[(H2)] There is a constant $D_2$ such that for every $s$ and $x \in L^s\cap B_4$
\begin{equation*}
	|\theta^s(x)| + |\beta^s(x)| \leq D_2(|x|^2 + 1).
\end{equation*}
where $\theta^s$ and $\beta^s$ are, respectively, the Lagrangian angle of $L^s$ and a primitive for the Liouville form on $L^s$.
\item[(H3)] The rescaled manifolds $\tilde{L}^s := (2s)^{-1/2}L^s$ converge in $C^{1,\alpha}_{loc}$ to $\Sigma$. Moreover the second fundamental form of $\tilde L^s$ is bounded uniformly in $s$ and without loss of generality we can assume that
\begin{equation*}
	\lim_{s\rightarrow 0} (\tilde{\theta}^s + \tilde{\beta}^s) = 0
\end{equation*}
locally on $\tilde L^s$. (Note that $\tilde L^s$ is exact in the ball $B_{4(2s)^{-1/2}}$ so we can make sense of $\tilde\beta^s$ in the limit.)
\item[(H4)] The connected components of $P \cap A(r_0\sqrt{s}, 4)$ are in one to one correspondence with the connected components of $L^s\cap A(r_0\sqrt{s}, 4)$, and each component can be parametrised as a graph over the corresponding plane $P_i$
\begin{equation*}
	L^s \cap A(r_0\sqrt{s},3) \subset \{x + u_s(x) | x \in P\cap A(r_0\sqrt{s}, 3)\} \subset L^s\cap A(r_0\sqrt{s}, 4) 
\end{equation*}
where the function $u_s:P\cap A(r_0\sqrt{s}, 3) \rightarrow P^\perp$ is normal to $P$ and satisfies the estimate
\begin{equation*}
	|u_s(x)| + |x|\left|\nb u_s(x)\right| + |x|^2|\nb^2 u_s(x)| \leq D_3\left(|x|^2 + \sqrt{2s}e^{-b|x|^2/2s}\right).
\end{equation*}
where $\nb$ denotes the covariant derivative on $P$.
\end{itemize}
We will denote by $(L^s_t)_{t\in[0,T_s)}$ a smooth solution of Lagrangian mean curvature flow with initial condition $L^s$. For $x_0 \in \mathbb{R}^{2n}$ and $t > 0$ we define
\begin{equation*}
	\Phi(x_0,t)(x) := \rho_{(x_0,0)}(x, -t) = \frac{1}{(4\pi t)^{n/2}}\mathrm{exp}\left(-\frac{|x-x_0|^2}{4t}\right)
\end{equation*}
We introduce a slightly modified notion of the Gaussian density ratios, which we will continue to refer to as the Gaussian density ratios, of $L^s_t$ at $x_0$, denoted $\Theta_t^s(x_0,r)$ and defined as 
\begin{equation}
	\Theta_t^s(x_0,r) := \int_{L^s_t} \Phi(x_0, r^2)d\mathcal{H}^n = \int_{L^s_t} \frac{1}{(4\pi r^2)^{n/2}}e^{-|x - x_0|^2/4r^2}d\mathcal{H}^n(x), \label{eq:densityratios}
\end{equation}
defined for $t < T_s$. The monotonicity formula of Huisken tells us that
\begin{equation*}
	\Theta_t(x_0,r) = \Theta(x_0,t+r^2,r) \leq \Theta(x_0,t+r^2,\rho) = \int_{L^s_{t + r^2 - \rho^2}} \Phi(x_0,t+r^2)d\mathcal{H}^n
\end{equation*}
for all $\rho \geq r$. In particular choosing $\rho^2 = t+r^2$ we have
\begin{equation*}
	\Theta_t(x_0,r) \leq \int_{L^s}\Phi(x_0,t+r^2)d\mathcal{H}^n.
\end{equation*}
We also define
\begin{equation*}
	\tilde L^s_t = \frac{L^s_t}{\sqrt{2(s+t)}}.
\end{equation*}
We will denote by $\tilde \Theta^s_t(x_0,r)$ the Gaussian density ratios of $(\tilde L^s_t)$, that is
\begin{equation*}
	\tilde\Theta^s_t(x_0,r) := \int_{\tilde L^s_t}\Phi(x_0,r)d\mathcal{H}^n.
\end{equation*}
One of the primary reasons for modifying the Gaussian density ratios is that our new ratios behave well under the above rescaling. Indeed we can calculate
\begin{equation*}
	\Theta^s_t(x_0,r) = \tilde\Theta^s_t\left(\frac{x_0}{\sqrt{2(s+t)}}, \frac{r}{\sqrt{2(s+t)}}\right).
\end{equation*}
The primary goal of this section is now to prove the following result. 

\begin{theorem}\label{thm:mainthm}
Let $\varepsilon_0 > 0$. There are $s_0$, $\delta_0$ and $\tau$ depending on $\alpha < 1$, $D_1, D_2, D_3$, $\Sigma$, $r_0$ and $\varepsilon$ such that if
\begin{equation*}
	t \leq \delta_0, \: r^2 \leq \tau t \:\: \text{and} \:\: s\leq s_0
\end{equation*}
then
\begin{equation*}
	\Theta_t^s(x_0,r) \leq 1 + \varepsilon_0
\end{equation*}
for every $x_0 \in B_1$.
\end{theorem}

We start by proving estimates like the one in the above theorem hold for a short time or far from the origin.

\begin{lemma}[Far from the origin estimate] \label{lem:fforigin}
Let $\varepsilon_0 > 0$. There are $\delta_1 > 0$, $K_0 < \infty$ such that if $r^2 \leq t \leq \delta_1$, then
\begin{equation*}
	\Theta^s_t(x_0, r) \leq 1 + \varepsilon_0
\end{equation*}
for all $x_0 \in A(K_0\sqrt{2t}, 1)$.
\end{lemma}
\begin{proof}
We first claim that there is a $K_0$ such that if $y_0 \in \mathbb{R}^{2n}$ has $|y_0| \geq K_0$ then for any $\lambda > 0$ and $s$ we have
\begin{equation*}
	\int_{\lambda(L^s\cap B_3(0))}\Phi(y_0,1)d\mathcal{H}^n \leq 1 + \varepsilon/2.
\end{equation*}
Indeed suppose that this were not the case, then there would exist sequences $y_i$, $\lambda_i$ and $s_i$ with $|y_i| \rightarrow \infty$ such that
\begin{equation}
	\int_{\lambda_i(L^{s_i}\cap B_3(0))} \Phi(y_i, 1) d\mathcal{H}^n \geq 1 + \varepsilon/2. \label{eq:shorttimeclaim1}
\end{equation}
First we note that $\lambda_i$ must be unbounded since
\begin{align*}
	\int_{\lambda_i(L^{s_i}\cap B_3(0))} \Phi(y_i, 1) d\mathcal{H}^n &\leq \int_{\lambda_i(L^{s_i}\cap B_3(0))}\frac{1}{(4\pi)^{n/2}}e^{-|y_i|^2/8}e^{3|x|^2/4}d\mathcal{H}^n \\
	&\leq e^{-|y_i|^2/8}\lambda_i^n\int_{L^{s_i}\cap B_3} \frac{1}{(4\pi)^{n/2}}e^{9\lambda_i^2/4}d\mathcal{H}^n \\
	&\leq C_i\lambda_i^n e^{-|y_i|^2/8 + c\lambda_i^2}\underbrace{\mathcal{H}^n(L^{s_i}\cap B_3(0))}_{\leq D_13^n}
\end{align*}
so it is easily seen that if $\lambda_i$ were bounded then \eqref{eq:shorttimeclaim1} would fail for large $i$. Next from the estimate (H4) we have that
\begin{equation*}
	|\nb^2 u^j_s(x)| \leq C\left(1 + \frac{\sqrt{2s}}{|x|^2}e^{-b|x|^2/2s}\right),
\end{equation*}
for every $x \in A(r_0\sqrt{2s}, 4)$ and hence
\begin{equation*}
	|A| \leq C\left(1 + \frac{1}{\sqrt{2s}}e^{-b|x|^2/2s}\right)
\end{equation*}
on $B_3\cap L^s$, since on $B_{r_0\sqrt{2s}}$ we have $|A| \leq C(2s)^{-1/2}$ where $C$ is a curvature bound for $\Sigma$. We rescale and define
\begin{equation*}
	\hat{L_i} := \lambda_iL^{s_i} \quad \quad \sigma_i := \lambda_i^2s_i,
\end{equation*}
so that on $\hat{L_i}$ we have the estimate
\begin{equation*}
	|A| \leq \frac{C}{\lambda_i}\left( 1 + \frac{1}{\sqrt{s_i}}e^{-b|x|^2/2\lambda_i^2s_i}\right) = C\left(\lambda_i^{-1} + \sigma_i^{-1/2} e^{-b|x|^2/2\sigma_i}\right).
\end{equation*}
Consequently $|A|\rightarrow 0$ uniformly on compact sets centred at $y_i$, so it follows that locally $\hat{L_i} - y_i $ converges to a plane, but this contradicts \eqref{eq:shorttimeclaim1}. \\
We next observe that (H1) ensures that we may choose $\delta_1 > 0$ small enough such that for any $x_0 \in B_1(0)$ and $l \leq 2\sqrt{\delta_1}$ we have
\begin{equation*}
	\int_{L^s\setminus B_3} \Phi(x_0, l) d \mathcal{H}^n \leq \varepsilon_0/2
\end{equation*}
By the monotonicity formula we have that for any $r^2, t \leq \delta_1$ 
\begin{align*}
	\Theta^s_t(x_0, r) &\leq \int_{L^s} \Phi(x_0, r^2 + t) d \mathcal{H}^n \\
	&= \int_{L^s\setminus B_3}\Phi(x_0, r^2 + t) d \mathcal{H}^n + \int_{L^s\cap B_3}\Phi(x_0, r^2 + t) d \mathcal{H}^n \\
	&\leq \varepsilon_0/2 + \int_{(r^2+t)^{-1}(L^s\cap B_3)}\Phi\left(\frac{x_0}{\sqrt{r^2 + t}},1\right)d\mathcal{H}^n \\
	&\leq 1 + \varepsilon_0
\end{align*}
provided that $|x_0| \geq K_0\sqrt{r^2 + t}$, so imposing the additional requirement that $r^2 \leq t$ this gives precisely the desired result.
\end{proof}

\begin{lemma}[Short-time estimate] \label{lem:shorttime}
Let $\varepsilon_0 > 0$. There are $s_1 > 0$ and $q_1 > 0$ such that if $s \leq s_1$, $r^2 \leq q_1 s$ and $t \leq q_1s$ then
\begin{equation}
	\Theta^s_t(x,r) \leq 1 + \varepsilon_0
\end{equation}
for all $x \in B_1$.
\end{lemma}

\begin{proof}
By lemma \ref{lem:fforigin} we need only prove the estimate for $x \in B_{K_0\sqrt{2t}}$. We apply lemma \ref{lem:c1alphaclose} with $R = K_0\sqrt{q_1} + 1$ where $q_1 = q_1(\Sigma, \varepsilon_0, \alpha)$ and the rescaled flow $\hat L_t := (2s)^{-1/2}L^s_{2st}$. This is a mean curvature flow with initial condition $\tilde L^s$. By (H3) we know that $\tilde L^s \rightarrow \Sigma$ in $C^{1,\alpha}_{loc}$, so in particular for $s$ small enough, we can ensure that $\tilde L^s$ is $\varepsilon(\varepsilon_0, \Sigma, \alpha)$-close to $\Sigma$ in $C^{1,\alpha}$. Hence for $r^2, t \leq q_1$ and $x \in B_{K_0\sqrt{q_1}}$ we have
\begin{equation*}
	\hat \Theta^s_t(x,r) = \int_{\hat L^s_t} \Phi(x, r^2) d\mathcal{H}^n = \int_{L^s_{2st}} \Phi(2sx, 2sr^2) d\mathcal{H}^n \leq 1 + \varepsilon_0,
\end{equation*}
or in other words
\begin{equation*}
	\Theta^s_t(x,r) \leq 1 + \varepsilon_0,
\end{equation*}
for all $r^2, t \leq q_1s$ and $x\in B_{K_0\sqrt{2sq_1}}$. However since $t\leq q_1s$ this holds for all $x \in B_{K_0\sqrt{2t}}$.
\end{proof}

The next lemma shows that in an annular region, and for short times, we retain control on both the distance to $P$ and the Gaussian density ratios that is uniform in $s$.

\begin{lemma}[Proximity to $P = P_1+P_2$]\label{lem:proximity}
There are constants $C_1$, and $r_1$ such that for any $\nu > 0$ there are $s_2, \delta_2 > 0$ such that the following holds. If $s \leq s_2$, $t \leq \delta_2$ and $r \leq 2$ then we have the estimates
\begin{equation*}
	\mathrm{dist}(y_0, P) \leq \nu + C_1e^{-|y_0|^2/C_1} \quad\quad \forall y_0 \in \tilde L^s_t \cap A(r_1, (s+t)^{-1/8})
\end{equation*}
and
\begin{equation*}
	\tilde \Theta^s_t(y_0, r) \leq 1 + \frac{\varepsilon_0}{2} + \nu \quad \quad \forall y_0 \in A(r_1, (s+t)^{-1/8}).
\end{equation*}
\end{lemma}

\begin{proof}
We consider $t\leq\delta_2$ and $s\leq s_2$ (both $\delta_2$ and $s_2$ to be chosen) and define
\begin{equation*}
	l := \frac{t}{2(s+t)} \quad\quad \Sigma^{(s,t)} := \frac{L^s}{\sqrt{2(s+t)}}.
\end{equation*}
Clearly $l \leq 1/2$ and also from (H4) we have that if $s_2$, $\delta_2$ are chosen small enough, then
\begin{equation*}
	\Sigma^{(s,t)}\cap A(r_0, 3(s+t)^{-1/8})
\end{equation*}
is graphical over $P \cap A(r_0, 3(s+t)^{-1/8})$. Moreover if $v_{(s,t)}$ is the function arising from this graphical decomposition then we have by scaling the estimate of (H4) that
\begin{align*}
	|v_{(s,t)}(x)| + &|x||\nb v_{(s,t)}(x)| + |x|^2|\nb^2v_{(s,t)}(x)| \\&\leq D_3\left(\sqrt{2(s+t)}|x|^2 + \left(\frac{\sqrt{2s}}{\sqrt{2(s+t)}}\right)e^{-2b(s+t)|x|^2/2s}\right) \\
	&\leq D_3\left(\sqrt{2(s+t)}|x|^2 + e^{-b|x|^2}\right).
\end{align*}
Let $c > 0$ be a constant that will be chosen later. If $s_2(D_3, r_0, c)$ and $\delta_2(D_3, r_0, c) > 0$ are small enough and $r_1(P,c) \geq \max\{r_0, 1\}$ is chosen to be large enough then we can ensure that
\begin{equation}
	|v_{(s,t)}(x)| + |x||\nb v_{(s,t)}(x)| \leq D_3\left(\sqrt{2(s+t)}|x|^2 + e^{-b|x|^2}\right) \leq c/2 \label{eq:closenessest}
\end{equation}
on $A(r_1, 3(s+t)^{-1/8})$. From now on we fix some $y_0 \in \tilde L^s_t\cap A\left(3r_1 + 1, (s+t)^{-1/8}\right)$, then $y_0\sqrt{2(s+t)}$ is a regular point of $(L^s_t)$ so by the monotonicity formula
\begin{equation*}
	1 \leq \Theta^s_0(y_0\sqrt{2(s+t)},\sqrt{t}) =  \int_{\Sigma^{(s,t)}}\Phi(y_0, l)d\mathcal{H}^n =: A + B + C
\end{equation*}
where
\begin{align*}
	A &:= \int_{\Sigma^{(s,t)}\setminus B_{3(s+t)^{-1/8}}} \Phi(y_0, l) d\mathcal{H}^n \\
	B &:= \int_{\Sigma^{(s,t)}\cap B_{r_1}} \Phi(y_0, l)d\mathcal{H}^n \\
	C &:= \int_{\Sigma^{(s,t)}\cap A(r_1, 3(s+t)^{-1/8})} \Phi(y_0, l)d\mathcal{H}^n
\end{align*}
We first estimate $A$. If $|x| \geq 3(s+t)^{-1/8} \geq 3|y_0|$ then
\begin{equation*}
	|x - y_0|^2 = |x|^2 - 2|x||y_0| + |y_0|^2 \geq |x|^2 - \frac{2|x|^2}{3} + |y_0|^2 = \frac{|x|^2}{3} + |y_0|^2,
\end{equation*}
so
\begin{equation*}
	\Phi(y_0, l) = \frac{1}{(4\pi l)^{n/2}}e^{-|x-y_0|^2/4l} \leq \frac{1}{(4\pi l)^{n/2}}e^{-|y_0|^2/4l}e^{-|x|^2/12l} = 3^{n/2}e^{-|y_0|^2/4l}\Phi(0,3l).
\end{equation*}
Therefore by choosing $C_1 = C_1(D_1, n)$ we can estimate
\begin{align*}
	A = \int_{\Sigma^{(s,t)}\setminus B_{3(s+t)^{-1/8}}} \Phi(y_0, l)d\mathcal{H}^n &\leq 3^{n/2} e^{-|y_0|^2/4l}\int_{\Sigma^{(s,t)}\setminus B_{3(s+t)^{-1/8}}} \Phi(0, 3l)d\mathcal{H}^n \\
	&\leq 3^{n/2}e^{-|y_0|^2/4l}\int_{(3l)^{-1/2}\Sigma^{(s,t)}} \Phi(0,1) d\mathcal{H}^n \\
	&\leq C_1 e^{-|y_0|^2/C_1},
\end{align*}
since $l$ is bounded independent of $s$ and $t$, and the estimate (H1) is scale invariant, so in particular is satisfied by $(3l)^{-1/2}\Sigma^{(s,t)}$. \\
Next we estimate $B$. Similarly as before we find that for $|x| \leq r_1 \leq |y_0|/3$ we have
\begin{equation*}
	|x-y_0|^2 \geq |x|^2 + \frac{|y_0|^2}{3}.
\end{equation*}
Thus
\begin{equation*}
	\Phi(y_0, l) \leq e^{-|y_0|^2/12}\Phi(0, l) \quad \quad \text{on   } B_{r_1}
\end{equation*}
hence by possibly increasing $C_1$ if necessary we have
\begin{equation*}
	B = \int_{\Sigma^{(s,t)}\cap B_{r_1}} \Phi(y_0, l) d\mathcal{H}^n \leq e^{-|y_0|^2/4l}\int_{\Sigma^{(s,t)}\cap B_{r_1}}\Phi(0,l)d\mathcal{H}^n \leq C_1e^{-|y_0|^2/C_1}.
\end{equation*}
Finally we deal with $C$. We denote by $a_i$ the orthogonal projection of $y_0$ onto $P_i$ and by $b_i$ the orthogonal projection of $y_0$ onto $P_i^\perp$. We suppose without loss of generality that
\begin{equation*}
	\mathrm{dist}(y_0, P) = |b_1|.
\end{equation*}
We will also denote by $\Sigma^{(s,t)}_i$ the component of $\Sigma^{(s,t)}\cap A(r_1, 3(s+t)^{-1/8})$ that is graphical over $\Pi_i := P_i\cap A(r_1, 3(s+t)^{-1/8})$, and by $v^i_{(s,t)}$ the corresponding function. Since we have that $P_1\cap P_2 = \{0\}$ it follows that for some $c = c(P) > 0$ we have that $|b_2| \geq c|y_0|$. Notice that since $|b_2| \leq |y_0|$ we have that $c \leq 1$. Suppose that $x \in \Sigma^{(s,t)}_2$, and denote by $x'$ the orthogonal projection onto $P_2$. Then we have
\begin{equation*}
	|y_0 - x|^2 = |a_2 + b_2 - x' - v^2_{(s,t)}(x')|^2 = |a_2 - x'|^2 + |b_2 - v^2_{(s,t)}(x')|^2.
\end{equation*}
Moreover by \eqref{eq:closenessest}, if $r_1$ is chosen large enough (and in particular larger than $1$),
\begin{equation*}
	|v^2_{(s,t)}(x')| \leq \frac{c}{2} \leq \frac{c|y_0|}{2}
\end{equation*}
so
\begin{equation*}
	|b_2 - v^2_{(s,t)}(x')| \geq |b_2| - |v^2_{(s,t)}(x')| \geq \frac{c|y_0|}{2}.
\end{equation*}
Consequently, defining $g_{ij} := \delta_{ij} + D_iv^2_{(s,t)}\cdot D_jv^2_{(s,t)}$, we can estimate
\begin{align*}
	\int_{\Sigma^{(s,t)}_2}\Phi(y_0,l) d\mathcal{H}^n &= \int_{\Pi_2} \frac{1}{(4\pi l)^{n/2}} \exp\left(\frac{-|a_2 - x'|^2-|b_2 - v_{(s,t)}^2(x')|^2}{4l}\right)\sqrt{\mathrm{det}(g_{ij})}d x' \\
	&\leq C e^{-c^2|y_0|^2/16l}\int_{P_2}\frac{1}{(4\pi l)^{n/2}}e^{-|a_2 - x'|^2/4l} dx' \\
	&\leq C_1e^{-|y_0|^2/C_1},
\end{align*}
where we used \eqref{eq:closenessest} to estimate the gradient terms arising in the surface measure. Combining this with the estimates for $A$ and $B$ we have that
\begin{equation}
	1 \leq \int_{\Sigma^{(s,t)}} \Phi(y_0, l) d\mathcal{H}^n \leq \int_{\Sigma^{(s,t)}_1} \Phi(y_0, l) d\mathcal{H}^n + C_1\exp\left(\frac{-|y_0|^2}{C_1}\right). \label{eq:sigma1lowerbound}
\end{equation}
Increasing $r_1$ for the last time if necessary, we can ensure that
\begin{equation*}
	C_1\exp\left(\frac{-|y_0|^2}{C_1}\right) \leq \frac{1}{2}.
\end{equation*}
Therefore we have that
\begin{equation*}
	\frac{1}{2} \leq \int_{\Sigma^{(s,t)}_1} \Phi(y_0,l)d\mathcal{H}^n \leq C\sup_{\Pi_1} \exp \left(-\frac{|b_1 - v^1_{(s,t)}|^2}{4l}\right).
\end{equation*}
Therefore it follows that $|b_1 - v^1_{(s,t)}|^2/4l$ is bounded on $\Pi_1$ independently of $l$, $s$ and $t$, thus we can estimate
\begin{equation*}
	\frac{|b_1 - v^1_{(s,t)}|^2}{4l} \leq C\left(1 - e^{-|b_1 - v^1_{(s,t)}|^2/4l}\right),
\end{equation*}
on $\Pi_1$ where $C$ is independent of $s$ and $t$. Moreover because the matrix $(D_iv^1_{(s,t)}\cdot D_jv^1_{(s,t)})$ has non-negative eigenvalues we have that
\begin{equation*}
	\sqrt{\det(\delta_{ij} + D_iv^1_{(s,t)}\cdot D_jv^1_{(s,t)})} \geq 1
\end{equation*}
so we can estimate
\begin{align*}
	\int_{\Pi_1} &\frac{|v^1_{(s,t)} - b_1|^2}{4l} \frac{\exp(-|x' - a_1|^2/4l)}{(4\pi l)^{n/2}}dx' \\ 
	&\leq C\int_{\Pi_1} \frac{\exp(-|x' - a_1|^2/4l)}{(4\pi l)^{n/2}}\left(1 - \exp\left(-\frac{|b_1 - v^1_{(s,t)}|^2}{4l}\right)\right)\sqrt{\det(g_{ij})}dx'\\
	&= C\left(\int_{\Pi_1} \frac{\exp(-|x' - a_1|^2/4l)}{(4\pi l)^{n/2}}\sqrt{\det(g_{ij})}dx' - \int_{\Sigma^{(s,t)}_1}\Phi(y_0,l)d\mathcal{H}^n\right) \\
	&\leq C\left(\int_{\Pi_1} \frac{\exp(-|x' - a_1|^2/4l)}{(4\pi l)^{n/2}}\sqrt{\det(g_{ij})}dx' - 1\right) + C_1\exp(-|y_0|^2/C_1) \\
	&\leq C\int_{\Pi_1} |\nb v^1_{(s,t)}| \frac{\exp(-|x'-a_1|^2/4l)}{(4\pi l)^{n/2}}dx' + C_1\exp(-|y_0|^2/C_1)
\end{align*}
where we used \eqref{eq:sigma1lowerbound} and the Taylor expansions for the square root and determinant functions as discussed in Lemma \ref{lem:graphicalest}. Therefore since
\begin{equation*}
	|b_1|^2 \leq (|b_1 - v^1_{(s,t)}| + |v^1_{(s,t)}|)^2 \leq 2(|b_1 - v^1_{(s,t)}|^2 + |v^1_{(s,t)}|^2)
\end{equation*}
we can estimate
\begin{equation*}
	|b_1|^2 \leq C_1\int_{\Pi_1} (|v^1_{(s,t)}|^2 + |\nb v^1_{(s,t)}|)\frac{\exp(-|x' - a_1|^2/4l)}{(4\pi l)^{n/2}}dx' + C_1\exp(-|y_0|^2/C_1).
\end{equation*}
We want to now control the integral terms on the right hand side. First we observe that $|a_1| \geq c|y_0|$ for some constant depending only on $P$. Moreover for any $0 \leq l \leq 1$ we have for any $x$, $a_1 \in \mathbb{R}^{2n}$
\begin{align*}
	b|x+a_1|^2 + \frac{|x|^2}{4l} &= |x|^2\left(\frac{1}{4l} + b\right) + |a_1|^2b + 2bx\cdot a_1 \\
	&\geq |x|^2\left(\frac{1}{4l} + b\right) + |a_1|^2b - \frac{8bl + 1}{8l}|x|^2 - \frac{8b^2l}{8bl+1}|a_1|^2 \\
	&\geq \frac{|x|^2}{8l} + \frac{b|a_1|^2}{8bl+1} 
\end{align*}
and hence for some $C_1 = C_1(D_1, D_3, P)$ we have
\begin{align*}
	\int_{\Pi_1} |\nb v^1_{(s,t)}|\frac{e^{-|x' - a_1|^2/4l}}{(4\pi l)^{n/2}}dx' &\leq C_1\int_{\Pi_1}\left(\sqrt{s+t}|x'| + e^{-b|x'|^2}\right)\frac{e^{-|x' - a_1|^2/4l}}{(4\pi l)^{n/2}}dx' \\
	&\leq C_1\sqrt{s+t} + D_3\int_{\mathbb{R}^n}e^{-b|x'|^2}\frac{e^{-|x' - a_1|^2/4l}}{(4\pi l)^{n/2}}dx' \\
	&\leq C_1\sqrt{s+t} + C_1\int_{\mathbb{R}^n}e^{-b|x'+a_1|^2} \frac{e^{-|x'|^2/4l}}{(4\pi l)^{n/2}}dx' \\
	&\leq C_1\sqrt{s+t} + C_1e^{-|a_1|^2/C_1}\int_{\mathbb{R}^n}\frac{e^{-|x'|^2/8l}}{(4\pi l)^{n/2}}dx' \\
	&\leq C_1\sqrt{s+t} + C_1e^{-|y_0|^2/C_1}.
\end{align*}
Similarly we can estimate
\begin{equation*}
	|v^1_{(s,t)}|^2 \leq C_1\left((t+s)|x|^4 + e^{-2b|x|^2}\right)
\end{equation*}
using \eqref{eq:closenessest}. So an entirely analogous calculation establishes the estimate
\begin{equation*}
	\int_{\Pi_1} |v^1_{(s,t)}|^2\frac{e^{-|x' - a_1|^2/4l}}{(4\pi l)^{n/2}}dx' \leq C_1\left((s+t) + e^{-|y_0|^2/C_1}\right).
\end{equation*}
Therefore choosing $s_2$ and $\delta_2$ depending on $D_1, D_2, P, r_0, \nu$ and $b$ we have that for all $s \leq s_2$ and $t \leq \delta_2$ we have
\begin{equation*}
	b_1 = \mathrm{dist}(y_0, P) \leq \nu + C_1e^{-|y_0|^2/C_1}
\end{equation*}
We next want to show that, possibly by increasing or decreasing $r_1$, $s_1$ and $\delta_1$ if necessary, that we also have the estimate
\begin{equation*}
	\tilde \Theta_t^s(y_0, r) \leq 1 + \frac{\varepsilon}{2} + \nu
\end{equation*}
for any $r \leq 2$. We have
\begin{align*}
	\tilde \Theta^s_t(y_0, r) &= \int_{\tilde L^s_t} \frac{1}{(4\pi r^2)^{n/2}}\exp\left(\frac{-|x-y_0|^2}{4r^2}\right)d\mathcal{H}^n \\
	&= \int_{L^s_t} \frac{1}{(4\pi(2(s+t))r^2)^{n/2}}\exp\left(\frac{-|x - \sqrt{2(s+t)}y_0|^2}{4r^2(2(s+t))}\right)d\mathcal{H}^n \\
	&= \Theta^s_{t}(\sqrt{2(s+t)}y_0, \sqrt{2(s+t)}r) \\
	&\leq \Theta^s_0(\sqrt{2(s+t)}y_0, \sqrt{2(s+t)r^2 + t}) \\
	&= \int_{L^s} \frac{1}{(4\pi(2(s+t)r^2 + t))^{n/2}}\exp\left(\frac{-|x -  \sqrt{2(s+t)}y_0|^2}{4(2(s+t)r^2 + t)}\right) d\mathcal{H}^n \\
	&= \int_{\Sigma^{(s,t)}} \frac{1}{(4\pi(r^2 + l))^{n/2}}\exp\left(\frac{-|x - y_0|^2}{4(l + r^2)}\right)d \mathcal{H}^n \\
	&= \int_{\Sigma^{(s,t)}} \Phi(y_0, l + r^2) d\mathcal{H}^n
\end{align*}
Therefore by splitting up the integral as before and estimating exactly analogously we have
\begin{align*}
	\tilde \Theta^s_{t}(y_0, r) &\leq \int_{\Sigma^{(s,t)}_1} \Phi(y_0, l + r^2)d\mathcal{H}^n + C_1 \exp\left(\frac{-|y_0|^2}{C_1}\right) \\
	&\leq \int_{\Pi_1} \frac{\exp\left(\frac{-|x' - a_1|^2}{4(l + r^2)} \right)}{(4\pi(l+r^2))^{n/2}}\sqrt{\mathrm{det}(\delta_{ij} + D_iv^1_{(s,t)}\cdot D_jv^1_{(s,t)})}dx' + C_1 \exp\left(\frac{-|y_0|^2}{C_1}\right) \\
	&\leq 1 + C_1\int_{\Pi_1}|\nb v^1_{(s,t)}| \frac{\exp\left(\frac{-|x' - a_1|^2}{4(l + r^2)} \right)}{(4\pi(l+r^2))^{n/2}}dx'  +C_1 \exp\left(\frac{-|y_0|^2}{C_1}\right) \\
	&\leq 1 + C_1\sqrt{s+t} + C_1\int_{\mathbb{R}^n} e^{-b|x'|^2} \frac{\exp\left(\frac{-|x' - a_1|^2}{4(l + r^2)} \right)}{(4\pi(l+r^2))^{n/2}} dx  + C_1 \exp\left(\frac{-|y_0|^2}{C_1}\right) \\
	&= 1 + C_1\sqrt{s+t} + C_1\int_{\mathbb{R}^n} e^{-b|x'+a_1|^2} \frac{\exp\left(\frac{-|x'|^2}{4(l + r^2)} \right)}{(4\pi(l+r^2))^{n/2}} dx  + C_1 \exp\left(\frac{-|y_0|^2}{C_1}\right)
\end{align*}
We want to estimate the exponential terms and pull out an exponential factor in $|a_1|$ so we estimate
\begin{align*}
	b|x+a_1|^2 + \frac{|x|^2}{4(l+r^2)} &\geq |x|^2\frac{4b(l+r^2)+1}{4(l+r^2)} + b|a_1|^2 - \frac{8b(l+r^2) + 1}{8(l+r^2)}|x|^2 \\
	&\quad\quad- \frac{8b^2(l+r^2)}{8b(l+r^2) + 1}|a_1|^2 \\
	&= \frac{|x|^2}{8(l+r^2)} + \frac{b|a_1|^2}{8b(l+r^2) + 1} \\
	&\geq \frac{|x|^2}{8(l+r^2)} + \frac{|a_1|^2}{C_1}
\end{align*}
where we used the fact that $l$ and $r$ are both bounded independently of $s$ and $t$. Therefore putting this together we have
\begin{align*}
	\tilde \Theta_t^s(y_0,r) &\leq 1 + C_1\sqrt{s+t} + C_1e^{-|a_1|^2/C_1}\int_{\mathbb{R}^n}\frac{e^{-|x|^2/8(l+r^2)}}{(4\pi(l+r^2))^{n/2}}dx + C_1e^{-|y_0|^2/C_1} \\
	&\leq 1 + C_1\sqrt{s+t} + C_1e^{-|y_0|^2/C_1}.
\end{align*}
Evidently an appropriate choice of $r_1$, $s_2$ and $\delta_2$ yields the required result.
\end{proof}

The following two Lemmas show that we have additional control in annular regions, specifically on normal deviation, curvature, Lagrange angle and the primitive for the Liouville form.

\begin{lemma} \label{lem:normaldev}
Let $F^s_t: L^s\rightarrow \mathbb{R}^{2n}$ be the normal deformation such that $L^s_t = F^s_t(L^s)$. We also define $\tilde F^s_t := (2(s+t))^{-1/2}F^s_t$ so that $\tilde L^s_t = \tilde F^s_t(L^s)$. Then there exist $r_2$, $\delta_3$, $s_3$ and $K < \infty$ such that if $t \leq \delta_3$ and $s \leq s_3$ then
\begin{equation*}
	\left|\tilde F^s_0(x) - \tilde F^s_t(x)\right| \leq K \quad\quad \text{whenever} \quad \quad \tilde F^s_0(x) \in A\left(r_2, (s+t)^{-1/8}/4\right).
\end{equation*}
\end{lemma}

\begin{proof}
By the proximity lemma \ref{lem:proximity} we may choose $r_2 \geq 1$, $\delta_3$ and $s_3$ such that if $t \leq \delta_3$ and $s \leq s_3$ then
\begin{equation*}
	\Theta^s_t(x,r) \leq 1 + \varepsilon
\end{equation*}
for all $r \leq 2\sqrt{2(s+t)}$ and $x \in A\left(r_2\sqrt{2(s+t)}, \sqrt{2}(s+t)^{3/8}\right)$. Hence by White's regularity theorem we can find a $C$ such that
\begin{equation*}
	\left|\frac{dF^s_t(p)}{dt}\right| = |\vec{H}| \leq \frac{C}{\sqrt{t}},
\end{equation*}
whenever $F^s_t(p) \in A\left(2r_2\sqrt{2(s+t)}, \sqrt{2(s+t)}(s+t)^{-1/8}/2\right)$. Therefore, choosing a larger $r_2$ and smaller $s_3$, $\delta_3$ if necessary we obtain by the fundamental theorem of calculus 
\begin{equation*}
	\left| F^s_t(p) - F^s_0(p)\right| \leq \int_0^t \frac{C}{\sqrt{s}}ds = 2C\sqrt{t}
\end{equation*}
whenever
\begin{equation*}
	F^s_0(p)\in A\left(r_2(2(s+t))^{1/2}, (2(s+t))^{1/2}(s+t)^{-1/8}/4\right)
\end{equation*}
which establishes the result.
\end{proof}

\begin{lemma} \label{lem:thetabetabound}
There are $\delta_4 > 0$ and $s_4 > 0$ such that for $0<s \leq s_4$ and $t<\delta_4$ 
\begin{equation}
	|A^s_t(x)| + |\theta^s_t(x)| + |\beta^s_t(x)| \leq D_4 \quad\quad \forall x \in L^s_t\cap A(1/3,3)
\end{equation}
\end{lemma}

\begin{proof}
The estimate is clearly true for $t = 0$ by assumption (H2). Moreover, by (H4) we can assume that for $s$ sufficiently small, each of the $L^s$ is the graph of a function with small gradient in the region $A(1/4, 4)$. Applying Lemma \ref{lem:graphicalest} we find that $L^s$ remains graphical with small gradient in $A(2/7,7/2)$ for some short time, which implies that $|\theta^s_t| \leq C$ for $\delta_4$ chosen small enough. \\
\indent That $|A^s_t|$ is bounded follows from Lemma \ref{lem:graphicalest} and Corollary \ref{cor:curvatureest}, since Lemma \ref{lem:graphicalest} implies small gradient for a short time, which allows use to apply Corollary \ref{cor:curvatureest} to get uniform curvature bounds for some short time in $A(1/3,3)$. \\
\indent Since $|\theta^s_t|$ and $|A^s_t|$ are both bounded, we have from the evolution equations of $\beta^s_t$ that
\begin{equation*}
	\left|\frac{d\beta^s_t}{dt}\right| \leq  \left|\langle Jx, \vec{H}\rangle\right| + 2|\theta^s_t| \leq C.
\end{equation*}
Hence for some suitable short time, $|\beta^s_t|$ also remains bounded in $A(1/3,3)$.
\end{proof}

The last of the technical lemmas in this section uses the monotonicity formula of Section \ref{sec:eveq} to show that after waiting for a short time dependent on $s$, we can find times at which the scaled flow $\tilde L^s_t$ is close to a self-expander in an $L^2$ sense. We later use this in the proof of the main theorem to get estimates on the density ratios via the stability result.

\begin{lemma} \label{lem:mon}
Let $a > 1$. Let $q_1$ be as given by Lemma \ref{lem:shorttime}, and set $q := q_1/a$. Then for all $\eta > 0$ and $R > 0$ there exist $\delta_5 > 0$, $s_5 > 0$ such that for all $s \leq s_5$ and $qs \leq T \leq \delta_5$ we have
\begin{equation*}
	\frac{1}{(a-1)T}\int_T^{aT}\int_{\tilde L^s_t\cap B_{R}}|\vec{H} - x^\perp|^2d\mathcal{H}^ndt \leq \eta.
\end{equation*}
\end{lemma}

\begin{proof}
Fix $R > 0$, $\eta > 0$. Suppose $s \leq s_5$ and $qs \leq T \leq \delta_5$, with $\delta_5$ and $s_5$ yet to be determined. Furthermore, we set $T_0 := R^2(s + aT) + aT$. Throughout the proof, we denote by $C$ a constant which depends on $a$, $R$ and $q$, but not on $T$ or $s$. We estimate
\begin{align}
	&\frac{1}{(a-1)T}\int_T^{aT}\int_{\tilde L^s_t\cap B_R}|\vec{H} - x^\perp|^2 d\mathcal{H}^ndt \nonumber\\
	&\quad\quad= \frac{1}{(a-1)T}\int_T^{aT}(2(s+t))^{-n/2 - 1}\int_{L^s_t\cap B_{R\sqrt{2(s+t)}}}|2(s+t)\vec{H} - x^\perp|^2d\mathcal{H}^ndt \label{eq:intest1}
\end{align}
Now supposing that $s_5$ and $\delta_5$ are small enough we can ensure that $R\sqrt{2(s+t)} \leq 2$. Moreover on $B_{R\sqrt{2(s+t)}}$ we have
\begin{align*}
	(T_0 - t)^{n/2}\rho_{0,T_0}(x,t) &= \frac{1}{(4\pi)^{n/2}}\exp\left(-\frac{|x|^2}{4(T_0 - t)}\right) \\
	&\geq \frac{1}{(4\pi)^{n/2}}\exp\left(-\frac{R^22(s+t)}{4(T_0 - t)}\right) \\
	&\geq \frac{1}{(4\pi)^{n/2}}\exp\left(-\frac{1}{2}\right)
\end{align*}
hence we continue estimating \eqref{eq:intest1} using the localized monotonicity formula of Lemma \ref{lem:localmon} ($\phi$ denotes the cut-off function given in that lemma which is $1$ on $B_2$ and $0$ outside of $B_3$)
\begin{align}
	\eqref{eq:intest1} &\leq \frac{C}{T}\int_T^{aT}(s+t)^{-(n+2)/2}(T_0 - t)^{n/2}\int_{L^s_t} \phi|2(s+t)\vec{H} - x^\perp|^2\rho_{0,T_0}d\mathcal{H}^ndt \nonumber\\
	&\leq \frac{C}{T}\int_T^{aT}(s+T)^{-(n+2)/2}(T_0 - T)^{n/2}\int_{L^s_t\cap A(2,3)}|\beta^s_t + 2(s+t)\theta^s_t|^2\rho_{0,T_0}d\mathcal{H}^ndt \nonumber\\
	&\quad+\frac{C}{T}(s+T)^{-(n+2)/2}(T_0-T)^{n/2}\int_{L^s_T}\phi|\beta^s_T + 2(s+T)\theta^s_T|^2\rho_{0,T_0}d\mathcal{H}^n \label{eq:intest2}
\end{align}
Now using the localized monotonicity a second time we have the estimate
\begin{equation*}
\frac{d}{dt}\int_{L^s_t}\phi|\beta^s_t + 2(s+t)\theta^s_t|^2\rho_{0,T_0}d\mathcal{H}^n \leq C\int_{L^s_t\cap A(2,3)}|\beta^s_t + 2(s+t)\theta^s_t|^2\rho_{0,T_0}d\mathcal{H}^n
\end{equation*}
so
\begin{align*}
	\int_{L^s_T}\phi|\beta^s_T+2(s+T)\theta^s_T|^2\rho_{0,T_0}d\mathcal{H}^n &\leq \int_{L^s_0}\phi|\beta^s_0+2s\theta^s_0|^2\rho_{0,T_0}d\mathcal{H}^n \\
	&\quad+ C\int_0^T\int_{L^s_t\cap A(2,3)}|\beta^s_t + 2(s+t)\theta^s_t|^2\rho_{0,T_0}d\mathcal{H}^ndt
\end{align*}
hence
\begin{align}
	\eqref{eq:intest2} &\leq \frac{C}{T}(s+T)^{-(n+2)/2}(T_0 - T)^{n/2}\int_{L^s_0}\phi|2s\theta^s_0 + \beta^s_0|^2\rho_{0,T_0}d\mathcal{H}^n \nonumber\\
	&\quad + \frac{C}{T}(s+T)^{-(n+2)/2}(T_0 - T)^{n/2}\int_0^{aT}\int_{L^s_t\cap A(2,3)}|2(s+t)\theta^s_t + \beta^s_t|^2\rho_{0,T_0}d\mathcal{H}^ndt. \label{eq:intest3}
\end{align}
Now $T_0 - T \leq C(s+T)$, with $C$ depending only on $R$ and $a$, so
\begin{align*}
	\eqref{eq:intest3} &\leq \frac{C}{T(s+T)}\int_{L^s_0}\phi|2s\theta^s_0 + \beta^s_0|^2\rho_{0,T_0}d\mathcal{H}^n \\
	&\quad + \frac{C}{T(s+T)}\int_0^{aT}\int_{L^s_t\cap A(2,3)}|2(s+t)\theta^s_t + \beta^s_t|^2\rho_{0,T_0}d\mathcal{H}^ndt\\
	&=: A + B
\end{align*}
We first estimate $B$, for which we make use of the estimate of Lemma \ref{lem:thetabetabound}
\begin{align}
	B &\leq \frac{C((s+aT) + 1)^2}{T(s+T)}\int_0^{aT}\int_{L^s_t\cap A(2,3)}\rho_{0,T_0}d\mathcal{H}^ndt \nonumber\\
	&\leq \frac{C((s+aT)+1)^2}{T(s+T)}\int_0^{aT}\int_{L^s_t\cap A(2,3)} |x|^4\rho_{0,T_0}d\mathcal{H}^ndt \nonumber\\
	&= \frac{C((s+aT)+1)^2}{T(s+T)}\int_0^{aT}(T_0 - t)^2\int_{(T_0 - t)^{-1/2}(L^s_t\cap A(2,3))} |x|^4\rho_{0,1}d\mathcal{H}^ndt \nonumber\\
	&\leq \frac{C((s+aT)+1)^2}{T(s+T)} T_0^3\sup_{t\in[0,aT]}\int_{(T_0 - t)^{-1/2}(L^s_t\cap A(2,3))} |x|^4\exp\left(-\frac{|x|^2}{4}\right) d\mathcal{H}^n \label{eq:intest4}
\end{align}
We note that $T_0 \leq (R^2(1/q + a) + a)T= CT$, $T_0 \leq C(s+T)$ and $T_0 \geq R^2(s + aT)$ so we can estimate
\begin{align*}
	\eqref{eq:intest4} &\leq C(T_0 + 1)^2T_0 \sup_{t\in[0,aT]}\int_{(T_0 - t)^{-1/2}L^s_t\cap A(2,3)} |x|^4\exp\left(-\frac{|x|^2}{4}\right)d\mathcal{H}^n \\
	&\leq C(T_0 + 1)^2T_0
\end{align*}
where we can estimate the supremum by a uniform constant because $L^s_t$ all have bounded area ratios with a uniform constant. Moreover $T_0 \leq R^2\delta_5(1/q + a) + a\delta_5$ so that by possibly decreasing $\delta_5$ we can ensure that $B \leq \eta/2$. \\
We next estimate $A$,
\begin{equation}
	A \leq \frac{C}{T(s+T)}\int_{L_0^s\cap B_3}|2s\theta^s_0 + \beta^s_0|^2\rho_{0,T_0}d\mathcal{H}^ndt \label{eq:Aest}
\end{equation}
First recall that if $\beta^s$ is primitive for the Liouville form on some $L^s$, then $\beta^s_l := l^{-2}\beta^s$ is primitive for the Liouville form on $l^{-1}L^s$. From here on we surpress the subscript $0$ of the $\beta^s$ and $\theta^s$ since we only ever integrate over the manifolds $L^s_0$, and we instead use a subscript $l$ to denote the rescaling factor of the $\beta^s$. We define
\begin{equation*}
	l := \sqrt{2(s+T)} \quad\quad \sigma := \frac{s}{s+T}
\end{equation*}
then
\begin{align*}
	\eqref{eq:Aest} &= \frac{C(s+T)}{T}\int_{l^{-1}(L^s_0\cap B_3)} |\sigma\theta^s + \beta^s_l|^2\rho_{0,l^{-2}T_0}d\mathcal{H}^ndt \\
	&\leq C \int_{l^{-1}(L^s_0\cap B_3)}|\sigma\theta^s + \beta^s_l|^2\rho_{0,l^{-2}T_0}d\mathcal{H}^n
\end{align*}
since $T \geq qs$, so we can absorb $(s+T)/T$ into the constant. Define
\begin{equation*}
	F(s,T) := \int_{l^{-1}(L^s_0\cap B_3)} |\sigma\theta^s + \beta^s_l|^2\rho_{0,l^{-2}T_0}d\mathcal{H}^n.
\end{equation*}
Notice that from the definition of $T_0$ we can find $C > 0$ independent of $T$ and $s$ such that $l^{-2}T_0 \in [C^{-1}, C]$. We want to show that by possibly again decreasing $s_5$ and $\delta_5$, we can ensure
\begin{equation*}
	F(s,T) \leq \eta/2.
\end{equation*}
Seeking a contradiction, suppose that this is not the case. Then we can find sequences $s_i$ and $T_i$ both converging to $0$ with $qs_i \leq T_i$ and such that
\begin{equation*}
	F(s_i,T_i) > \eta/2.
\end{equation*}
After possibly extracting a subsequence which we don't relabel, we may assume that $l_i^{-2}T_0 \rightarrow T_1$. We split the rest of the proof into two cases. \\
\textbf{Case 1:} Suppose that (after possibly extracting a further subsequence) we have that $\sigma_i \rightarrow \sigma > 0$. Then by (H3) we have
\begin{equation*}
	l^{-1}_iL^{s_i}_0 = \sigma_i^{1/2}\tilde L^{s_i}_0 \rightarrow \sigma^{1/2}\Sigma
\end{equation*}
in $C^{1,\alpha}$. Therefore we have
\begin{align*}
	\lim_{i\rightarrow \infty} F(T_i, s_i) &= \lim_{i\rightarrow \infty}\int_{\sigma_i^{1/2}\tilde L^{s_i}_0\cap l_i^{-1}B_3}|\sigma_i\theta^{s_i} + \beta^{s_i}_{l_i}|^2\rho_{0,l_i^{-2}T_0}d\mathcal{H}^n \\
	&= \lim_{i\rightarrow\infty} \sigma_i^2 \int_{\tilde L^{s_i}_0\cap (2s_i)^{-1/2} B_3} |\tilde \theta^{s_i} + \tilde \beta^{s_i}|^2\rho_{0,l_i^{-2}\sigma_i^{-1}T_0}d\mathcal{H}^n = 0 \\
\end{align*}
because $|\tilde \theta^{s_i} + \tilde \beta^{s_i}|$ is bounded by $D_2(1 + |x|^2)$ on $B_{3(2s_i)^{-1/2}}$, which means that since $l_i^{-2}\sigma_i^{-1}T_0 \rightarrow \sigma^{-1}T_1 > 0$ the contribution to the integral outside some fixed large ball is small uniformly in $i$. Moreover by (H3) we have $\lim_{i\rightarrow \infty}|\tilde \theta^{s_i} + \tilde \beta^{s_i}|^2 = 0$ locally, so inside this large ball the integral can be made as small as desired.\\
\textbf{Case 2:} Suppose now that, again after possibly passing to a not relabelled subsequence, $\sigma_i \rightarrow 0$. Then
\begin{align*}
	\lim_{i\rightarrow\infty}\int_{l_i^{-1}(L^{s_i}_0\cap B_{r_0\sqrt{s_i}})}&|\sigma_i\theta^{s_i} + \beta^{s_i}_{l_i}|^2\rho_{0,l_i^{-2}T_0}d\mathcal{H}^n\\ &= \lim_{i\rightarrow\infty}\int_{\sigma_i^{1/2}\tilde L^{s_i}_0\cap B_{r_0\sqrt{\sigma_i/2}}} |\sigma_i\theta^{s_i} + \beta^{s_i}_{l_i}|^2\rho_{0,l_i^{-2}T_0}d\mathcal{H}^n \\
	&= \lim_{i\rightarrow \infty}\sigma_i^2\int_{\tilde L^{s_i}_0\cap B_{r_0/\sqrt{2}}} |\tilde \theta^{s_i} + \tilde \beta^{s_i}|^2 \rho_{0,\sigma_i^{-1}l_i^{-1}T_0}d\mathcal{H}^n = 0
\end{align*}
because $|\tilde \theta^{s_i} + \tilde \beta^{s_i}|^2 \rightarrow 0$ locally, and $\rho$ is bounded. So to estimate $\lim_{i\rightarrow\infty}F(T_i,s_i)$ we need only control the integral in the annulus $A(r_0\sqrt{\sigma_i/2},3l_i^{-1})$. We first notice that by (H4), provided $i$ is large enough, $l_i^{-1}L^{s_i}\cap A(r_0\sqrt{\sigma_i/2}, 3l_i^{-1})$ is graphical over $P$, and if $v_i$ is the function arising from this decomposition we have the estimate
\begin{equation*}
	|v_i(x')| + |x'||\nb v_i(x')| + |x'|^2|\nb^2v_i(x')| \leq D_3\left(l_i|x'|^2 + \sigma_i^{1/2} e^{-b|x'|^2/2\sigma_i}\right).
\end{equation*}
In the graphical region, the normal space to the graph is spanned by the vectors $n_j := (-\nb v^j_i, e_j)$ for $j = 1,\dots,n$ where $e_j$ denotes the vector in $\mathbb{R}^n$ whose $j$th entry is $1$, and all other entries are $0$, and $v^j_i$ is the $j$th coordinate of $v_i$. Then given an orthonormal basis for the normal space $\nu_1,\dots,\nu_n$ we have $\nu_j = \sum_{k=1}^n\alpha_{jk}n_k$ so it follows that
\begin{equation*}
	|x^\perp| \leq C\sum_{j=1}^n |\langle x, n_j\rangle|
\end{equation*}
where $C$ depends only on the $\alpha_{jk}$. Now
\begin{align*}
	\langle x, n_j\rangle &= \langle(x', v_i(x')), (-\nb v^j_i, e_j)\rangle \\
	&= -\langle x', \nb v^j_i(x')\rangle + v^j_i(x')
\end{align*}
from which it follows that
\begin{equation*}
	|x^\perp| \leq C\left(|v_i(x')| + |x'||\nb v_i(x')|\right).
\end{equation*}
Therefore
\begin{equation}
	|\nabla \beta^{s_i}_{l_i}| = |x^\perp| \leq C\left(l_i|x'|^2 + \sigma_i^{1/2}\right)\label{eq:betaest}.
\end{equation}
Using this estimate we can control $\beta^{s_i}_{l_i}$ independently of $i$ on the annular region $A(r_0\sqrt{\sigma_i/2}, 3l_i^{-1}) \cap l_i^{-1}L^{s_i}$. Indeed suppose that $x \in A(r_0\sqrt{\sigma_i/2}, 3l_i^{-1}) \cap l_i^{-1} L^{s_i}$, then there is a corresponding $x' \in A(r_0\sqrt{\sigma_i/2}, 3l_i^{-1}) \cap P$ such that $x = x' + v_i(x')$. Define 
\begin{equation*}
	x_i' := \frac{r_0\sqrt{\sigma_i}}{\sqrt{2}|x'|} x' \quad \text{ and }\quad x_i := x_i' + v_i(x_i').
\end{equation*}
Note that $x_i$ of course depends on the original choice of $x$ as well as $i$. We may now define a curve in $\tilde L^{s_i}$ by setting
\begin{equation*}
	\gamma(t) := x_i' + t(x' - x_i') + v_i(x_i' + t(x' - x_i')).
\end{equation*}
By the fundamental theorem of calculus we can write
\begin{align*}
	\beta^{s_i}_{l_i}(x) &= \beta^{s_i}_{l_i}(x_i) + \int_0^1 \frac{d}{dt}\beta^{s_i}_{l_i}(\gamma(t))dt \\
	&\leq \beta^{s_i}_{l_i}(x_i) + \int_0^1|\nabla \beta^{s_i}_{l_i}(\gamma(t))||\gamma'(t)|dt
\end{align*}
now
\begin{equation*}
	|\gamma'(t)| \leq |x' - x_i'| + |\nb v_i||x' - x_i'| \leq C|x|
\end{equation*}
so
\begin{align*}
	\beta^{s_i}_{l_i}(x) &\leq \beta^{s_i}_{l_i}(x_i) + C|x|\int^1_0 l_i|x_i' + t(x' - x_i')|^2 + \sigma_i^{1/2}dt \\
	&\leq \beta^{s_i}_{l_i}(x_i) + C(l_i|x|^3 + \sigma_i^{1/2}).
\end{align*}
Now $\beta^{s_i}_{l_i}(x_i) = \sigma_i\tilde\beta^{s_i}(\sigma^{1/2}_ix_i)$, moreover since $|x_i|$ is bounded independently of $i$ or the original choice of $|x|$ we have
\begin{equation*}
	\lim_{i\rightarrow \infty} \tilde\beta^{s_i}(\sigma_i^{1/2}x_i) + \tilde \theta^{s_i}(\sigma_i^{1/2}x_i) = 0
\end{equation*}
uniformly in $x$. Thus
\begin{equation*}
	\lim_{i\rightarrow \infty} \beta^{s_i}_{l_i}(x_i) = -\lim_{i\rightarrow \infty}\sigma_i\tilde\theta^{s_i}(\sigma_i^{1/2}x_i) = 0
\end{equation*}
uniformly in $x$ as $\tilde\theta^{s_i}$ is bounded and $\sigma_i \rightarrow 0$. Therefore we may bound the term $\beta^{s_i}_{l_i}(x_i)$ by some sequence $b_i$ with $b_i \rightarrow 0$. Consequently we have the estimate
\begin{equation*}
	|\beta^{s_i}_{l_i}(x)| \leq C\left(l_i|x|^3 + \sigma_i^{1/2}|x|\right) + b_i
\end{equation*}
on $A(r_0\sqrt{\sigma_i/2}, 3l_i^{-1}) \cap l_i^{-1}$, hence
\begin{align*}
	\lim_{i\rightarrow\infty} F(T_i,s_i) &= \lim_{i\rightarrow\infty} \int_{l_i^{-1}L^{s_i}\cap A(r_0\sqrt{\sigma/2}, 3l_i^{-1})} |\sigma_i\theta^{s_i} + \beta^{s_i}_{l_i}|^2\rho_{0,l_i^{-2}T_0}d\mathcal{H}^n \\
	&= \lim_{i\rightarrow\infty}\int_{l_i^{-1}L^{s_i}\cap A(r_0\sqrt{\sigma/2}, 3l_i^{-1})} |\beta^{s_i}_{l_i}|^2\rho_{0,l_i^{-2}T_0}d\mathcal{H}^n \\
	&\leq \lim_{i\rightarrow \infty} C(l_i^2 + \sigma_i + b_i^2)\int_{l_i^{-1}L^{s_i}}(|x|^6 + |x|^2 + 1)\rho_{0,l_i^{-2}T_0}d\mathcal{H}^n = 0,
\end{align*}
where we again used the fact that $l_i^{-2}T_0 \rightarrow T_1 > 0$, so that outside of some large ball the contribution to the integral is very small. This limit being zero is a contradiction, so we are done. 
\end{proof}

We may now embark on the proof of Theorem \ref{thm:mainthm}. Changing scale, to prove the main theorem it would in fact suffice to show the following (which is very slightly stronger due to the bound on the scale of the density ratios)
\begin{theorem*}[Rescaled main theorem]
There exist $s_0$, $\delta_0$ and $\tau$ such that if $t \leq \delta_0$, $r^2 \leq \tau$ and $s \leq s_0$, then
\begin{equation*}
	\tilde \Theta^s_t(x_0, r) \leq 1 + \varepsilon_0
\end{equation*}
for all $x_0$ with $|x_0| \leq (2(s+t))^{-1/2}$. 
\end{theorem*}
If we set $\tau := q_1/(2(q_1 + 1))$ where $q_1$ is as in the short-time existence lemma, then the rescaled version of the same lemma tells us that
\begin{lemma*}[Rescaled short-time existence]
If $s \leq s_1$, $t \leq q_1 s$ and $r^2 \leq \tau$ then
\begin{equation*}
	\tilde \Theta^s_t(y_0, r) \leq 1 + \varepsilon_0
\end{equation*}
$|y_0| \leq (2(s+t))^{-1/2}$.
\end{lemma*}
and similarly the rescaled far from the origin estimate tells us that
\begin{lemma*}[Rescaled far from origin]
If $r^2 \leq \tau$ and $q_1s \leq t \leq \delta_1$
\begin{equation*}
	\tilde \Theta^s_t(y_0, r) \leq 1 + \varepsilon_0
\end{equation*}
whenever $K_0 \leq |y_0| \leq (2(s+t))^{-1/2}$.
\end{lemma*}
Thus to prove the rescaled main theorem, it suffices to show that for appropriately chosen $s_0, \delta_0$ and $\tau$ the following holds true: if $r^2 \leq \tau$, $s \leq s_0$, $t \leq \delta_0$ and $t \geq q_1 s$ then
\begin{equation*}
	\tilde \Theta^s_t(y_0, r) \leq 1 + \varepsilon_0
\end{equation*}
whenever $|y_0| \leq K_0$. This is what we now show.

\begin{proof}[Proof of Theorem \ref{thm:mainthm}]
Define for each
\begin{equation*}
	T_s := \sup\left\{ T | \tilde \Theta^s_t(y_0, r) \leq 1 + \varepsilon_0 \:\: \forall r^2 \leq \tau, t \leq T, |y_0| \leq K_0 \right\}.
\end{equation*}
We now claim that we can find $\delta_0 > 0$ and $s_0 > 0$ such that $T_s \geq \delta_0$ for all $s \leq s_0$. Indeed, with $\tau$ defined as before, we choose $a > 1$ with $a < (1 + 2\tau)$. Let $C$ be the constant of Brian White's local regularity theorem, and set 
\begin{equation*}
	\tilde C := C\frac{\sqrt{2(a+3)}}{\sqrt{q_1(a-1)}}.
\end{equation*}
We next let $r_3 := \max\{r_0, r_1, r_2, 1\}$, where $r_0$, $r_1$, and $r_2$ are as in, respectively, the construction of the approximating family, Lemma \ref{lem:proximity}, and Lemma \ref{lem:normaldev}. Let $R := \sqrt{1+ 2q_1}K_0 + r_3$, and $\varepsilon = \varepsilon(\Sigma, \varepsilon_0, \alpha)$ as given by Lemma \ref{lem:c1alphaclose}. We apply the stability result, Theorem \ref{thm:stability} with $R = R$; $r = r_3$; $C = \max\{C_1, C\}$ the constants from Lemma \ref{lem:proximity}, and the construction of the approximating family respectively; $M = \tilde C$; $\tau = \tau$; $\Sigma = \Sigma$ and $\varepsilon = \varepsilon$. Thus we obtain $\tilde R \geq R$, $\eta > 0$ and $\nu \geq 0$ as in the theorem. Apply Lemma \ref{lem:mon} with $\eta = \eta/2$ and $R = \tilde R$. This gives $s_5$ and $\delta_5$ such that the lemma holds. Next apply Lemma \ref{lem:proximity} with $\nu$ to obtain $s_2$ and $\delta_2$. We now let $s_0 := \min\{s_1,s_2,s_3,s_4,s_5\}$ and $\delta_0 := \min\{\delta_1, \delta_2, \delta_3, \delta_4, \delta_5\}$. We finally possibly decrease $s_0$ and $\delta_0$ slightly to ensure that
\begin{equation*}
	(s_0 + \delta_0)^{-1/8} \geq 2\tilde R.
\end{equation*}
This will ensure that in the annular region $A(r_3, \tilde R)$ we have all of the estimates of the intermediate lemmas of this section. We now claim that these $s_0$ and $\delta_0$ are the required constants. Specifically we claim that for all $s \leq s_0$ we have $T_s \geq \delta_0$. Indeed, suppose that this were not the case and that for some $s \leq s_0$ we have $T_s < \delta_0$. Define $T := T_s/a$, then since $T < T_s$ we have for all $t \in [T, T_s)$
\begin{equation*}
	\tilde \Theta^s_t(x,r) \leq 1 + \varepsilon_0
\end{equation*}
for all $r^2 \leq \tau$ and $x \in B_{K_0}$. In fact, as has already been observed, the same is true for all $|x| \leq (s+t)^{-1/8}$, so in particular for all $|x| \leq 2\tilde R$. Let $\hat L^s_l$ denote the Lagrangian mean curvature flow with initial condition $\tilde L^s_T$. Then it is easy to verify that
\begin{equation*}
	\hat L^s_l = \sqrt{1 + 2l}\tilde L^s_{T + \sigma^2l},
\end{equation*}
where $\sigma^2 = 2(s + T)$. This implies the density ratio control
\begin{equation*}
	\hat\Theta^s_l(x,r) \leq 1 + \varepsilon_0,
\end{equation*}
for all $l$ such that $T + \sigma^2l \in [T, T_s)$, $r^2 \leq \tau$ and $x \in B_{2\tilde R}$. By the local regularity theorem of Brian White this means we get curvature bounds of the form
\begin{equation*}
	|\hat A^s_l | \leq \frac{C}{\sqrt{l}} \quad \quad l \leq \tau, \text{  on  } B_{\tilde R},
\end{equation*}
or, scaled back to the original scale this means
\begin{equation*}
	|A^s_t| \leq \frac{C}{\sqrt{t - T}},
\end{equation*}
on $B_{\sigma \tilde R}$ for all $t < T_s$ with $T \leq t \leq T + 2(s+T)\tau = (1 + 2\tau)T + 2\tau s$. Notice in particular that
\begin{equation*}
	T_s = aT \leq aT_0 \leq (1 + 2\tau)t_0 + 2 \tau s,
\end{equation*}
so the above estimate always holds up to time $T_s$. Let $t_0 := T(a+1)/2$. Then 
\begin{alignat*}{3}
	|A^s_{t_0}| &\leq \frac{C\sqrt{2}}{\sqrt{(a-1)T}} &&= \frac{C\sqrt{\frac{2(a+3)}{a-1}}}{\sqrt{(a + 3)T}} \\
	&= \frac{C\sqrt{\frac{2(a+3)}{a-1}}}{\sqrt{2(t_0 + T)}} &&\leq \frac{C\sqrt{\frac{2(a+3)}{a-1}}}{\sqrt{2(t_0 + q_1s)}} \\
	&\leq \frac{\tilde C}{\sqrt{2(t_0 + s)}},
\end{alignat*}
on $B_{\sigma \tilde R}$, where we $\tilde C$ is defined as before. Similarly, if $t > 0$ is such that $t_0 + t \leq T_s$ then
\begin{alignat*}{3}
	|A^s_{t_0 + t}| &\leq \frac{C\sqrt{2}}{\sqrt{(a-1)T + t}} &&\leq \frac{C\sqrt{\frac{2(a+3)}{a-1}}}{\sqrt{(a+3)T + 2t}} \\
	&\leq \frac{C\sqrt{\frac{2(a+3)}{a-1}}}{\sqrt{2(t_0 + T + t)}} &&\leq \frac{\tilde C}{\sqrt{2(t_0 + t + s)}}.
\end{alignat*}
In other words, for each $t \in [t_0, T_s)$ we have
\begin{equation*}
	|A^s_t| \leq \frac{\tilde C}{\sqrt{2(s+t)}} \quad \text{  on  } B_{\sigma \tilde R},
\end{equation*}
which implies that for each $t \in [t_0, T_s)$ we have
\begin{equation*}
	|\tilde A^s_t| \leq \tilde C
\end{equation*}
on $B_{\tilde R}$. Applying Lemma \ref{lem:mon}, we may select $t_1 \in [t_0, T_s)$ with 
\begin{equation*}
	\int_{\tilde L^s_{t_1}\cap B_{\tilde R}} |\vec{H} - x^\perp|^2d\mathcal{H}^n \leq \eta.
\end{equation*}
Condition $(iv)$ of Theorem \ref{thm:stability} holds for $\tilde L^s_{t_1}$ by Lemma \ref{lem:proximity}, and condition $(ii)$ holds by definition of $T_s$ as $t_1 < T_s$. Hence by Theorem \ref{thm:stability} we know that $\tilde L^s_{t_1}$ is $\varepsilon$-close to $\Sigma$ in $C^{1,\alpha}(B_{\tilde R})$. Redefine $\hat L^s_l$ to be the Lagrangian mean curvature flow with initial condition $\tilde L^s_{t_1}$, then Lemma \ref{lem:c1alphaclose} says that
\begin{equation*}
	\hat \Theta^s_l(x,r) \leq 1 + \varepsilon_0 \quad\quad r^2, \: l \leq q_1
\end{equation*}
for $|x| \leq \tilde R - 1$. By definition of $\tilde R$ this means that the same is true for $|x| \leq \sqrt{1 + 2q_1}K_0$. Rescaling this is equivalent to
\begin{equation*}
	\tilde \Theta^s_{t_1 + 2(s+t_1)l}\left(\frac{x}{\sqrt{1 + 2l}}, \frac{r}{\sqrt{1 + 2l}}\right) \leq 1 + \varepsilon_0 
\end{equation*}
for $r^2$, $l \leq q_1$ and $|x| \leq \sqrt{1 + 2q_1}K_0$. Or in other words
\begin{equation*}
	\tilde \Theta^s_t(x,r) \leq 1 + \varepsilon_0
\end{equation*}
for $r^2 \leq q_1/(1 + 2q_1) = \tau$, $|x| \leq K_0$ and $t_1 \leq t \leq (1 + 2q_1)t_1 + 2q_1s$. However, $(1 + 2q_1)t_1 + 2q_1s > at_1 > aT = T_s$, a contradiction.
\end{proof}

%% file: shorttime.tex
In this section we prove the following short time existence result using Theorem \ref{thm:mainthm}.
\begin{theorem}
Suppose that $L\subset\mathbb{C}^n$ is a compact Lagrangian submanifold of $\mathbb{C}^n$ with a finite number of singularities, each of which is asymptotic to a pair of transversally intersecting planes $P_1 + P_2$ where neither $P_1 + P_2$ nor $P_1 - P_2$ are area minimizing. Then there exists $T > 0$ and a Lagrangian mean curvature flow $(L_t)_{0<t<T}$ such that as $t \searrow 0$, $L_t \rightarrow L$ as varifolds and in $C^\infty_{loc}$ away from the singularities.
\end{theorem}
\begin{proof}
For simplicity we suppose that $L$ has only one singularity at the origin. The case where $L$ has more than one follows by entirely analogous arguments. By standard short time existence theory for smooth compact mean curvature flow, for all $s \in (0,c]$ there exists a Lagrangian mean curvature flow $(L^s_t)_{0 \leq t \leq T_s}$ with $T_s > 0$. We claim that there exists a $T_0 > 0$ such that $T_s \geq T_0$ for all $s$ sufficiently small, and that furthermore, we have interior estimates on $|A|$ and its higher derivatives for all $t > 0$, which are independent of $s$. By virtue of Lemma \ref{lem:graphicalest}, we can apply Corollary \ref{cor:curvatureest} on small balls everywhere outside $B_{1/3}$ to get uniform curvature bounds outside of $B_{1/2}$ up to time $\min\{T_s,\delta\}$ where $\delta > 0$ is independent of $s$. Uniform estimates on the higher derivatives then immediately follow by standard parabolic PDE theory.\\
To obtain the desired bounds on $B_{1/2}$ we use Theorem \ref{thm:mainthm}. Let $\varepsilon_0 > 0$ be the constant of Brian White's local regularity theorem. Then Theorem \ref{thm:mainthm} says that there exist $s_0$, $\delta_0$ and $\tau$ such that for all $s \leq s_0$, $t \leq \delta_0$ and $r^2 \leq \tau t$ we have
\begin{equation*}
	\Theta^s_t(x_0, r) = \Theta^s(x, t+r^2, r) \leq 1 + \varepsilon_0.
\end{equation*}
This implies that for all $s \leq s_0$, $t\leq \delta_0$ and $r^2 \leq \tau t$ we have $\Theta^s(x, t, r) \leq 1 + \varepsilon$. We now fix $s \leq s_0$, $t_0 < \min\{\delta_0, T_s\}$, and $\rho \leq \min\{1/4, \sqrt{t_0}\}$. Then it follows that $B_{2\rho}(x_0) \subset B_1$, and furthermore that 
\begin{equation*}
	\Theta^s(x,t,r) \leq 1 + \varepsilon_0
\end{equation*}
for all $r \leq \tau \rho^2$, and $(x,t) \in B_{2\rho}(x_0) \times (t_0 - \rho^2 , t_0]$. Then it immediately follows from White's theorem that
\begin{equation*}
	|A| \leq \frac{C}{\sqrt{t - t_0 + \rho^2}}
\end{equation*}
for all $(x,t) \in B_\rho(x_0) \times (t_0 - \rho^2, t_0]$, where $C$ depends only on $\tau$ and $\varepsilon_0$. These estimates are then uniform in $s$ for $s \leq s_0$. Moreover, these curvature bounds, along with those outside of the ball $B_{1/2}$, imply that $T_s \geq \min\{\delta, \delta_0\}$.\\
Because the estimates are independent of $s$, they pass to the limit in the varifold topology when we take a subsequential limit of the flows and so we obtain a limiting flow $(L_t)_{0<t<T_0}$, for which $L_t \rightarrow L$ as varifolds.\\
Note that away from the singularities, we can obtain uniform curvature estimates on $|A|$ thanks to Corollary \ref{cor:curvatureest}, so it follows that $(L_t)$ attains the initial data $L$ in $C^\infty_{loc}$ away from the singular points.
\end{proof}

%% file: approx.tex
In this section, we consider a Lagrangian submanifold $L$ of $\mathbb{C}^n$ with a singularity at the origin which is asymptotic to the pair of planes $P$ considered in Section \ref{sec:stability}. We approximate $L$ by gluing in the self-expander $\Sigma$ which is asymptotic to $P$ at smaller and smaller scales in place of the singularity. We will show that this yields a family of compact Lagrangians, exact in $B_4$, which satisfy the hypotheses (H1)-(H4) given in Section \ref{sec:main} which are required to implement the analysis in that section. \\
Since $L$ is conically singular we may write $L\cap B_4$ as a graph over $P\cap B_4$ (possibly rescaling $L$ so that this is the case). We may further apply the Lagrangian neighbourhood theorem (its extension to cones was proved by Joyce, \cite[Theorem 4.1]{joyce2003special}), so that we may identify $L\cap B_4$ with the graph of a one-form $\gamma$ on $P$. Recall that the manifold corresponding to the graph of such a one-form is Lagrangian if and only if the one-form is closed.\\
Moreover, since we have assumed that $L$ is exact inside $B_4$, there exists $u\in C^\infty(P\cap B_4)$ such that $\d u=\gamma$. Since we know that $\gamma$ must decay quadratically, we can choose a primitive for $\gamma$ which has cubic decay, i.e.,  
\begin{equation}\label{icest}
 |\nabla^ku(x)|\le C|x|^{3-k}.
\end{equation}
We saw in Theorem \ref{thm:seuniqueness} that there exists a unique, smooth zero-Maslov self-expander asymptotic to $P$. We may also identify the self-expander outside a ball of radius $r_0$ with the graph of a one-form over $P$ and, since a zero-Maslov class Lagrangian self-expander is globally exact, there exists a function $v\in C^\infty(P\backslash B_{r_0})$ such that the self-expander is described by the exact one-form $\psi=\d v$ on $P\backslash B_{r_0}$. Further, Lotay and Neves proved \cite[Theorem 3.1]{LotNe13}
\begin{equation}\label{selfexest}
 \|v\|_{C^k(P\backslash B_r)}\le Ce^{-br^2} \quad \text{for all }r\ge r_0.
\end{equation}
We will glue $\Sigma_s:=\sqrt{2s}\Sigma$ into the initial condition $L$ to resolve the singularity. Our new manifold, $L^s$, will be the rescaled self-expander $\Sigma^s$ inside $B_{r_0\sqrt{2s}}$, the manifold $L$ outside $B_{4}$ and will smoothly interpolate between the two on the annulus $A(r_0\sqrt{2s},4)$.\\
To do this, we will glue together the primitives of the one-forms corresponding to these manifolds, before taking the exterior derivative. This gives us a one-form that will describe $L^s$ on the annulus $A(r_0\sqrt{2s},4)$, which ensures $L^s$ is still Lagrangian and is exact in $B_4$. We will then show that this family satisfies the properties (H1)-(H4).\\
Let $\varphi:\mathbb{R}_+\to [0,1]$ be a smooth function satisfying $\varphi\equiv 1$ on $[0,1]$ and $\varphi\equiv 0$ on $[2,\infty)$. Consider the one-form given by, for $r_0\sqrt{2s}\le |x|\le 4$, $0<s\le c$
\begin{align}\label{oneform}
 \gamma_s(x)&=\d w_s(x)=\d\left[\varphi(s^{-1/4}|x|)2sv(x/\sqrt{2s})+(1-\varphi(s^{-1/4}|x|))u(x)\right], 
\end{align}
 where we have that $r_0\sqrt{2s}<s^{1/4}<2s^{1/4}<4$ holds for all $s\le c$. Notice that in particular we must have $c<1$. Then $\gamma_s(x)\equiv \psi_s(x):=\sqrt{2s}\psi(x/\sqrt{2s})$, the one-form corresponding to the rescaled self-expander $\Sigma_s$ for $|x|<s^{1/4}$ and $\gamma_s\equiv \gamma$ for $|x|>2s^{1/4}$. Notice that since $\gamma_s$ is exact, it is closed and therefore its graph corresponds to an exact Lagrangian.\\
We define $L^s$ by 
\begin{itemize}
 \item $L^s\cap B_{r_0\sqrt{2s}}=\Sigma_s\cap B_{r_0\sqrt{2s}}$,
 \item $L^s\cap A(r_0\sqrt{2s},4)=$graph $\gamma_s$,
 \item $L^s\backslash B_4=L\backslash B_4$.
\end{itemize}
We will now show that $L^s$ satisfies (H1)-(H4).\\
For (H1), notice that both the self-expander and the initial condition individually satisfy (H1), and so for the rescaled self-expander, we have that
\begin{align*}
 \mathcal{H}^n(\Sigma_s\cap B_R)&=\mathcal{H}^n((\sqrt{2s}\Sigma)\cap B_R)=(2s)^{n/2}\mathcal{H}^n(L\cap B_{R/\sqrt{2s}}) \\
&\le (2s)^{n/2}D_1\left(\frac{R}{\sqrt{2s}}\right)^n=D_1R^n.
\end{align*}
Since $L^s$ interpolates between $\Sigma_s$ and $L$ on a compact region, $L^s$ satisfies (H1).\\
We see that (H2) is satisfied because the Lagrangian angle of the initial condition $L$ and the self-expander  $\Sigma$ are bounded, as is that of the rescaled self-expander $\Sigma_s$ by Lemma \ref{lem:eveq} (i) and the maximum principle, since the Lagrangian angle of $P$ is locally constant. When we interpolate between the two, we may consider the formula for the Lagrangian angle of a Lagrangian graph, as seen in \cite[pg. 5]{chau2012lagrangian}. This tells us that a Lagrangian graph in $\mathbb{C}^n$ (over $\mathbb{R}^n$) given by $(x_1,...,x_n,u_1(x),...,u_n(x))$, where $u:\mathbb{R}^n\to \mathbb{R}, u_i:=\frac{\partial u}{\partial x_i},$ has Lagrangian angle
$$
\theta=\sum \arctan \lambda_i,
$$
where the $\lambda_i$'s are the eigenvalues of the Hessian of $u$. Since the eigenvalues of the Hessian of $u$ are some non-linear function of the second derivatives of $u$, if the $C^2$ norm of $u$ is small we have that the Lagrangian angle of the graph is close to that of the Lagrangian angle of the plane that $u$ is a graph over. So we can uniformly bound the Lagrangian angle of the graph. Since in our case, the Lagrangian angle of $\gamma_s$ is given by the sum of arctangents of the eigenvalues of the Hessian of the function $w_s$, and, as we will show when we prove (H4), the $C^2$ norm of $w_s$ is small, this means that we can uniformly bound the Lagrangian angle of the graph $\gamma_s$, and so the Lagrangian angle of $L^s$.  \\ 
On the initial condition, since $\lambda=Jx$, we have that $\d \beta_L=\lambda|_L=(Jx)^T$. Therefore, $\beta_L$ is bounded quadratically, and so is the primitive for the Liouville form of $L^s\backslash B(2s^{1/4})$.
On the self-expander, applying the maximum principle to Lemma \ref{lem:eveq} (ii), we have $\beta_s$ (the primitive of $\lambda|_{\Sigma_s}$) is bounded by $\beta_P$, and so $|\beta^s(x)|\le |\beta_P(x)|\le C|x|^2$ for $|x|<s^{1/4}$. So it remains to check this still holds where we interpolate. We perform a calculation similar to that in the proof of Lemma \ref{lem:eveq}(ii).
We have that, for $L^s_t$ the manifold described by the graph of the one-form $t\d w_s$,
$$
\frac{\d}{\d t} \lambda|_{L^s_t}=:\mathcal{L}_{J\nabla w_s} \lambda|_{L^s_t}=\d(J\nabla w_s\lrcorner \lambda|_{L^s_t})+J\nabla w_s\lrcorner \d\lambda|_{L^s_t}.
$$
Since $\d\lambda=\omega$ and $J\nabla w_s\lrcorner \omega=\d w_s$ and possibly adding constant to $\beta^s_t$ dependent on $s$ and $t$, we have that
$$
\frac{\d\beta^s_t}{\d t}=-2w_s+\langle x,\nabla w_s \rangle|_{L^s_t},
$$
where $\d \beta^s_t$ is equal to the restriction of the Liouville form $\lambda$ to graph of $t\gamma_s$. Integrating, we find that
$$
 \beta^s=\beta_P-2w_s+\int_0^1\langle x,\nabla w_s\rangle|_{L^s_t} \,\d t,
$$ 
where $\beta_P$ is the primitive for $\lambda$ on $P$. 
Now, $w_s$ is bounded independently of $s$ by $D(1+|x|^2),$ using \eqref{icest} and \eqref{selfexest}, as is $\langle x,\nabla w_s\rangle$, using Cauchy-Schwarz and the estimates \eqref{icest} and \eqref{selfexest} so we find that $\beta^s$ is bounded independently of $s$ on the annulus $A(s^{1/4},2s^{1/4})$. Therefore, we have that
$$
|\theta^s(x)|+|\beta^s(x)|\le D_2(|x|^2+1).
$$
and so (H2) is satisfied.\\

To show that (H3) is satisfied, recall that we define $L^s$ as $L^s\cap B_{r_0\sqrt{2s}}=\Sigma_s\cap B_{r_0\sqrt{2s}}$, $L^s\backslash B_4=L\backslash B_4$ and we interpolate smoothly between the two, which exactly happens when $s^{1/4}\le|x|\le2s^{1/4}$. Therefore when we rescale by $1/\sqrt{2s}$, we have that $\tilde L^s\cap B_{r_0}\equiv \Sigma$.  So it remains to check convergence outside this ball.\\
On the annulus $r_0\le|x|\le4/\sqrt{2s}$, $\tilde L^s$ is identified with the graph of the following one-form 
$$
\tilde \gamma_s (x)=\d \left[\varphi(s^{1/4}|x|)v(x)+(1-\varphi(s^{1/4}|x|))\frac{u(\sqrt{2s}x)}{2s}\right].
$$
From this expression, noticing that 
$$
\frac{u(\sqrt{2s}x)}{2s}\le C\frac{(2s)^{3/2}x^3}{2s}=C\sqrt{2s}x,
$$
 we see that as $s\to 0$, $\tilde\gamma_s\to \d v=\psi$, the one-form whose graph is identified with $\Sigma$. This says that, outside $B_{r_0}$, $\tilde L^s\to \Sigma$ as $s\to 0$ smoothly. Therefore we actually have stronger than the required $C^{1,\alpha}_{loc}$ convergence. \\
Finally, we check that the second fundamental form of $\tilde L^s$ is uniformly bounded in $s$. We have that the second fundamental form of $\Sigma$ must be bounded, and if $A$ is the second fundamental form of $L$, rescaling $L$ by $1/\sqrt{2s}$ means that the second fundamental form scales by $\sqrt{2s}$. Since $\sqrt{2s}<1$, we can uniformly bound both second fundamental forms so that $\tilde L^s$, which is a combination of both $\Sigma$ and $1/\sqrt{2s}L$, has second fundamental form uniformly bounded in $s$. 
 \\
To see (H4), first notice that since we can write $L^s\cap A(r_0\sqrt{2s},4)$ as a graph over $P\cap A(r_0\sqrt{2s},4)$, we have that $L^s$ has the same number of connected components as $P$ in the annulus $A(r_0\sqrt{2s},4)$. \\ 
We now must estimate $\gamma_s$. Firstly, note that we have
\begin{equation}\label{rescaledest}
 |\nabla^k(v(x/\sqrt{2s}))|\le |(2s)^{-k/2}(\nabla^kv)(x/\sqrt{2s})|\le C(2s)^{-k/2}e^{-b|x|^2/2s},
\end{equation}
where we have used \eqref{selfexest}.\\
We will need different estimates on $2s\nabla^{2}v(x/\sqrt{2s})$ and $2s\nabla^{3}v(x/\sqrt{2s})$, which we find as follows.
\begin{align}\nonumber
 |2s\nabla^2v(x/\sqrt{2s})|&\le Ce^{-b|x|^2/2s}=C\frac{\sqrt{2s}}{|x|}\frac{|x|}{\sqrt{2s}}e^{-b|x|^2/2s}  
\\
&=C\frac{\sqrt{2s}}{|x|}e^{-\tilde b |x|^2/2s}\frac{|x|}{\sqrt{2s}}e^{-\tilde b |x|^2/2s} 
\le \tilde C \frac{\sqrt{2s}}{|x|}e^{-\tilde b |x|^2/2s},\label{nabla2}
\end{align}
where $\tilde b=b/2$ and $\tilde C= Ce^{-1/2}/\sqrt{b}$, since the function $y\mapsto ye^{-by^2/2}$ is bounded independently of $y$ (by $e^{-1/2}/\sqrt{b}$) on $\mathbb{R}$, and so $\tilde C$ is independent of $s$.\\
A similar calculation, this time noticing the uniform boundedness of the function $y\mapsto ye^{-by/2}$ for $y>0$ we can show that
\begin{equation}\label{nabla3}
 |2s\nabla^3 v(x/\sqrt{2s})|\le C \frac{\sqrt{2s}}{|x|^2}e^{-b|x|^2/2s},
\end{equation}
where we make $C$ (which remains independent of $s$) larger if necessary and $b$ smaller (which does not affect the previous estimates).\\
We have, using the definition in \eqref{oneform},
\begin{align*}
 |\gamma_s|=|\nabla w_s|&=|\varphi'(s^{-1/4}|x|)2s^{3/4}v(x/\sqrt{2s})+\varphi(s^{-1/4}|x|)2s\nabla[v(x/\sqrt{2s})] \\
&-s^{-1/4}\varphi'(s^{-1/4}|x|)u(x)+(1-\varphi(s^{-1/4}|x|))\nabla u(x)|,
\end{align*}
and, using that $s^{3/4}=\sqrt{s}s^{1/4}< \sqrt{s}$ since $s<1$, \eqref{icest} and \eqref{rescaledest} imply that
\begin{align}\nonumber
 |\gamma_s|&\le \sqrt{2s}Ce^{-b|x|^2/2s}+\sqrt{2s}Ce^{-b|x|^2/2s}+C|x|^{3-1}+C|x|^2 \\
&\le C\left[\sqrt{2s}e^{-b|x|^2/2s}+|x|^2\right],\label{star1}
\end{align}
where we have made $C$ larger.\\
Now consider
\begin{align*}
 |\nabla \gamma_s|=|\nabla^2 w_s|&=|\varphi''(s^{-1/4}|x|)2s^{1/2}v(x/\sqrt{2s})+\varphi'(s^{-1/4}|x|)4s^{3/4}\nabla[v(x/\sqrt{2s})] \\
&+\varphi(s^{-1/4}|x|)2s\nabla^2[v(x/\sqrt{2s})]-s^{-1/2}\varphi''(s^{-1/4}|x|)u(x)\\
&-2s^{-1/4}\varphi'(s^{-1/4}|x|)\nabla u(x) +(1-\varphi(s^{-1/4}|x|))\nabla^2u(x)|
\end{align*}
Using that on the support of $\varphi'$ and $\varphi''$ we have ($s<1$) $\sqrt{s}<s^{1/4}\le \sqrt{2}\sqrt{2s}/|x|$, and applying the estimates \eqref{rescaledest} and \eqref{nabla2}
\begin{align}\nonumber
 |\nabla \gamma_s|&\le C\left[\left(\frac{\sqrt{2s}}{|x|}+\frac{\sqrt{2s}}{|x|}+\frac{\sqrt{2s}}{|x|}\right)e^{-b|x|^2/2s}+|x|^{3-2}+|x|^{2-1}+|x|\right] \\
&\le C\left[\frac{\sqrt{2s}}{|x|}e^{-b|x|^2/2s}+|x|\right].\label{star2}
\end{align}
Finally, performing a similar computation to those above and combining \eqref{rescaledest}, \eqref{nabla2} and \eqref{nabla3} we find that
\begin{equation}\label{star3}
|\nabla^2\gamma_s|\le C\left[\frac{\sqrt{2s}}{|x|^2}e^{-b|x|^2/2s}+1\right].
\end{equation}
Combining \eqref{star1}, \eqref{star2} and \eqref{star3}, we have that 
$$
|\gamma_s|+|x||\nabla \gamma_s|+|x|^2|\nabla^2\gamma_s|\le D_3\left(|x|^2+\sqrt{2s}e^{-b|x|^2/2s}\right),
$$
where $D_3$ is a constant independent of $s$. Therefore (H4) is satisfied.

%% file: appendix.tex
We collect in the appendix a few technical results about Mean curvature flow in high codimension that were used throughout the paper. The first is a graphical estimate. Specifically, if the initial manifold can be written locally as a graph with small gradient in some cylinder, then the submanifold remains graphical in a smaller cylinder and we retain control on the gradient. To state this more rigorously we first introduce some notation. The notation and statement of the result are as in \cite{IlNeSc14}. Given any point $x\in\mathbb{R}^{n+k}$ we write $x = (\hat x, \tilde x)$, where $\hat x$ is the projection onto $\mathbb{R}^n$ and $\tilde x$ is the projection onto $\mathbb{R}^k$. We define the cylinder $C_R(x_0)\subset \mathbb{R}^{n+k}$ by
\begin{equation*}
	C_r(x) = \{x\in\mathbb{R}^{n+k}| |\hat x - \hat x_0| < r, |\tilde x - \tilde x_0| < r\}.
\end{equation*}
Furthermore, we write $B^n_r(x_0) = \{(\hat x, \tilde x_0)\in\mathbb{R}^{n+k}| |\hat x - \hat x_0| < r\}$.
\begin{lemma} \label{lem:graphicalest}
Let $(M^n_t)_{0\leq t<T}$ be a smooth mean curvature flow of embedded $n$-dimensional submanifolds in $\mathbb{R}^{n+k}$ with area ratios bounded by $D$. Then for any $\eta > 0$, then there exists $\varepsilon$, $\delta > 0$, depending only on $n$, $k$, $\eta$, $D$, such that if $x_0\in M_0$ and $M_0\cap C_1(x_0)$ can be written as $\mathrm{graph}(u)$, where $u:B^n_1(x_0)\rightarrow\mathbb{R}^k$ with Lipschitz constant less than $\varepsilon$, then
\begin{equation*}
	M_t\cap C_\delta(x_0) \quad\quad t\in [0,\delta^2)\cap [0,T)
\end{equation*}
is a graph over $B^n_\delta(x_0)$ with Lipschitz constant less than $\eta$ and height bounded by $\eta\delta$. 
\end{lemma}
The proof can be found in \cite{IlNeSc14}

Next we prove that if an initial manifold $M$ is close to some smooth manifold $\Sigma$ in $C^{1,\alpha}$, then one gets estimates on the density ratios that are independent of $M$.

\begin{lemma}\label{lem:c1alphaclose}
Let $\Sigma$ be a smooth manifold with bounded curvature and let $(M_t)_{t\in[0,T)}$ be a solution of mean curvature flow. Fix $\varepsilon_0 > 0$, $\alpha < 1$. There are $\varepsilon = \varepsilon(\Sigma, \varepsilon_0, \alpha) > 0$ and $q_1 = q_1(\Sigma, \varepsilon_0, \alpha) > 0$ such that for every $R \geq 2$, if $M_0$ is $\varepsilon$-close to $\Sigma$ in $C^{1,\alpha}(B_R)$ then for every $r^2,t\leq q_1$ and $y\in B_{R-1}$ we have
\begin{equation*}
	\Theta_t(y,r) \leq 1 + \varepsilon_0
\end{equation*}
\end{lemma}
\begin{proof}
This follows immediately from Lemma \ref{lem:graphicalest}. Indeed the curvature bound on $\Sigma$ means that there is a uniform radius $r$ such that for any $x\in\Sigma$, $\Sigma\cap C_r(x)$ is (after maybe rotating) a graph with small gradient. By requiring that $\varepsilon$ is small enough we can therefore ensure that any $M_0$ which is $\varepsilon$-close to $\Sigma$ in $C^{1,\alpha}(B_r(x))$ is also a graph with small gradient. It only remains to apply Lemma \ref{lem:graphicalest}.
\end{proof}

\subsection{Local curvature estimates for high codimension graphical MCF}
In \cite{EckHu89} Ecker and Huisken proved celebrated curvature estimates for entire graphs moving by mean curvature in codimension one, they then localised these in \cite{EckHu91} to prove interior estimates for hypersurfaces moving by mean curvature flow. Analogous results in higher codimension have been proved by Mu-Tao Wang in \cite{Wang01} and \cite{Wang04} respectively. In light of examples of Lawson and Osserman \cite{LawsonOsserman} one needs to assume an additional `$K$ local Lipschitz condition', such a condition is in fact satisfied by any $C^1$ manifold at small enough scales, so for our purposes there will be no problems applying the estimates. We would like to use the estimates derived in \cite{Wang04} without the time localisation, so we will briefly outline the changes to the proof, though all calculations remain analogous to those used by Wang or Ecker-Huisken. We first introduce the notation used by Wang in \cite{Wang01, Wang04}. We consider a function $u:\mathbb{R}^n\rightarrow\mathbb{R}^k$ whose graph in $\mathbb{R}^n\times\mathbb{R}^k$ evolves by mean curvature. $\Omega$ is the volume form of $\mathbb{R}^n$ which we can extend to a parallel form on $\mathbb{R}^n\times\mathbb{R}^k$. As shown by Wang, one can calculate that 
\begin{equation*}
	\ast \Omega = \frac{1}{\sqrt{\mathrm{det}(\delta_{ij} + D_i u\cdot D_ju)}} = \frac{1}{\sqrt{\prod_{i=1}^n(1 + \lambda_i^2)}}	
\end{equation*}
where $\lambda_i$ are the eigenvalues of $\sqrt{(du_t)^Tdu_t}$. Moreover, for $\varepsilon > 0$ small (depending only on the dimensions $n$ and $k$), we have that if 
\begin{equation*}
	\mathrm{det}(\delta_{ij} + D_iu\cdot D_ju) < 1 + \varepsilon,
\end{equation*}
(this is precisely the $K$ local Lipschitz condition of \cite{Wang04} with $K = 1/(1+\varepsilon)$) then $\ast\Omega$ satisfies the evolution inequality
\begin{equation*}
	\frac{d}{dt}\ast \Omega \geq \Delta \ast \Omega + \frac{1}{2}\ast \Omega|A|^2.
\end{equation*}
To simplify notation slightly we define $\eta := \ast \Omega$, then one can estimate (following \cite{Wang01})
\begin{equation*}
	\frac{d}{dt}\eta^p \geq \Delta\eta^p + \left(\frac{p}{2} - p(p-1)n\varepsilon\right)\eta^p|A|^2.
\end{equation*}
We also recall the evolution of the second fundamental form under mean curvature flow yields the differential inequality
\begin{equation*}
	\frac{d}{dt}|A|^2 \leq \Delta |A|^2 - 2|\nabla|A||^2 + C|A|^4,
\end{equation*}
where $C$ is a dimensional constant. We see that these estimates precisely tell us that we are in the correct setting to apply Lemma 4.1 of \cite{Wang04} with the choices $h = |A|$ and $f = \eta^p$. Following the proof of Lemma 4.1 we find that with $\varphi$ defined as
\begin{equation*}
	\varphi(x) := x/(1-\kappa x)
\end{equation*}
with $\kappa > 0$ to be determined, we have the following evolution inequality for $g = \varphi(\eta^{-2p})|A|^2$
\begin{align*}
	\left(\frac{d}{dt} - \Delta\right)g &\leq -2C\kappa g^2 - \frac{2\kappa}{(1 - \kappa\eta^{-2p})^2}|\nabla \eta^{-p}|^2g -2\varphi\eta^{3p}\nabla\eta^{-p}\cdot\nabla g.
\end{align*}
We then introduce the cut-off function $\xi := (R^2 - r)^2$ where $R > 0$ is a fixed radius and $r(x,t)$ satisfies
\begin{equation*}
	\left|\left(\frac{d}{dt} - \Delta\right)r\right| \leq c(n,k) \quad\quad |\nabla r|^2 \leq c(n,k)r,
\end{equation*}
then following \cite{EckHu91} we arrive at
\begin{align*}
	\left(\frac{d}{dt} - \Delta\right)g\xi &\leq -C\kappa\xi g^2 - 2(\varphi\eta^{3p}\nabla\eta^{-p} + \xi^{-1}\nabla\xi) \cdot \nabla(g\xi) \\
	&\quad\quad+ c(n,k)\left(\left(1 + \frac{1}{\kappa\eta^{-2p}}\right)r + R^2\right)g.
\end{align*}
It is possible now to also localise in time as in \cite{EckHu91}, which would get us to the estimates in \cite{Wang04}, but for our purposes this is unnecessary, so instead we now suppose that $m(T) := \sup_{0\leq t\leq T} \sup_{\{x\in M_t|r(x,t)\leq R^2\}} g\xi$ is attained at some time $t_0 > 0$, then at a point where $m(T)$ is attained we have
\begin{equation*}
	C\kappa\xi g^2 \leq c(n,k)\left(1 + \frac{1}{\kappa\eta^{-2p}}\right)R^2g
\end{equation*}
multiplying by $\xi/C\kappa$ we have
\begin{equation*}
	m(T) \leq \frac{c(n,k)}{C\kappa}\left(1 + \frac{1}{\kappa\eta^{-2p}}\right)R^2
\end{equation*}
In the set $\{x\in M_t|r(x,t)\leq \theta R^2, \: t\in[0,T]\}$ we have $\varphi \geq 1$, $\xi \geq (1-\theta)^2R^4$ so
\begin{equation*}
	|A|^2(1-\theta)^2R^4 \leq g\xi\leq \frac{c(n,k)}{C\kappa}\left(1 + \frac{1}{\kappa\eta^{-2p}}\right)R^2
\end{equation*}
We now choose
\begin{equation*}
	\kappa := \frac{1}{2}\inf_{\{x\in M_t| r(x,t) \leq R^2 \: t\in [0,T]\}} \eta^{2p},
\end{equation*}
then as $\eta^{-2p} \geq 1$ and $\kappa \leq 1/2$ we have that $(1 + 1/\kappa \eta^{-2p}) \leq 2/\kappa$, so
\begin{equation*}
	|A|^2 \leq \frac{c(n,k)}{\kappa^2R^2(1-\theta)^2} = \frac{c(n,k)}{R^2(1-\theta)^2}\sup_{\{x\in M_t| r\leq R^2 \: t\in[0,T]\}} \eta^{-4p}
\end{equation*}
In the set $\{x\in M_t|r(x,t)\leq \theta R^2, \: t\in[0,T]\}$. The preceding discussion establishes the following theorem
\begin{theorem}[High codimension interior estimate]\label{thm:curvatureest}
Let $R > 0$ and suppose that $K_{R^2} := \{(x,t)\in M_t | r(x,t) \leq R^2 \}$ is compact and can be written as a graph over some plane for $t\in[0,T]$. Suppose further that if the graph function is denoted by $u$, that
\begin{equation*}
	\mathrm{det}(\delta_{ij} + D_iu\cdot D_ju) < 1 + \varepsilon,
\end{equation*} 
where $\varepsilon > 0$ depends only on $n$ and $k$. Then for any $t\in[0,T]$ and $\theta \in (0,1)$ we have
\begin{align}\label{eq:curvest}
	\sup_{K_{\theta R^2}}&|A|^2 \leq \max\left\{ \frac{c(n)}{R^2(1-\theta)^2}\sup_{K_{R^2}} \eta^{-4p}, \sup_{\{x\in M_0| r\leq R^2\}} \frac{|A|^2\varphi(\eta^{-2p})}{(1-\theta)^2} \right\}
\end{align}
\end{theorem}
If we denote by $\cdot^T$ projection onto the plane over which $M_t$ is graphical, then it's easy to see that
\begin{equation*}
	\left(\frac{d}{dt} - \Delta\right)|x^T| = 0
\end{equation*}
for $x = F(p,t)$ some point in $M_t$. Therefore, defining $r(x,t) := |x^T|^2$ we have
\begin{align*}
	\left|\left(\frac{d}{dt} - \Delta\right) r\right| = 2|(\nabla x)^T|^2 &\leq c(n,k) \\
	|\nabla r|^2 = 4|x^T|^2|(\nabla x)^T|^2 &\leq c(n,k)r.
\end{align*}
With this choice of $r$ we have the following corollary
\begin{corollary}\label{cor:curvatureest}
Under the assumptions of Theorem \ref{thm:curvatureest}, with the particular choice $r(x,t) = |x^T|^2$ we have the estimate
\begin{equation}\label{eq:curvatureest2}
	\sup_{B_{\theta R}(y_0)\times[0,T]} |A|^2 \leq \min\left\{\frac{c(n,k)}{R^2(1-\theta)^2}\sup_{B_R(y_0)\times[0,T]}\eta^{-4p}, \sup_{\{B_R(y_0)\times\{0\}\}} \frac{|A|^2\varphi(\eta^{-2p})}{(1-\theta)^2}\right\}
\end{equation}
where $B_R(y_0)$ denotes a ball centred at $y_0$ with radius $R$ in the plane.
\end{corollary}